\documentclass[reqno,a4paper,12pt]{amsart} 

\usepackage{amsmath,amscd,amsfonts,amssymb}
\usepackage{mathrsfs,dsfont}

\allowdisplaybreaks[3]

\numberwithin{equation}{section}
\numberwithin{figure}{section}

\newcommand\R{\mathbb{R}}

\newcommand\Z{\mathbb{Z}}

\newcommand\G{\mathcal{G}}

\newcommand\del{\delta}
\newcommand\Del{\Delta}
\newcommand\lam{\lambda}
\newcommand\Lam{\Lambda}

\newcommand\Om{\Omega}

\newcommand\1{\mathds{1}}
\newcommand\eps{\varepsilon}
\newcommand{\pphi}{\varphi}

\newcommand{\PP}{\mathbb{P}}
\newcommand{\E}{\mathbb{E}}

\newcommand{\TTT}{T}
\newcommand{\SSS}{S}
\newcommand{\ts}{\tau}

\newcommand{\gtsp}{g_{ts}^{\phantom{*}}}
\newcommand{\gtss}{g_{ts}^*}

\renewcommand\le{\leqslant}
\renewcommand\ge{\geqslant}
\renewcommand\leq{\leqslant}

\newcommand\sbt{\subset}

\newcommand{\ft}[1]{\widehat{#1}}
\newcommand{\dotprod}[2]{\langle #1 , #2 \rangle}

\newcommand{\supp}{\operatorname{supp}}

\newcommand{\cspan}{\overline{\operatorname{span}}}

\theoremstyle{plain}
\newtheorem{thm}{Theorem}[section]
\newtheorem{lem}[thm]{Lemma}
\newtheorem{lemma}[thm]{Lemma}

\newtheorem*{claim*}{Claim}

\newcommand{\thmref}[1]{Theorem~\ref{#1}}
\newcommand{\secref}[1]{Section~\ref{#1}}

\newcommand{\lemref}[1]{Lemma~\ref{#1}}

\theoremstyle{definition}
\newtheorem{definition}[thm]{Definition}
\newtheorem*{definition*}{Definition}
\newtheorem*{remarks*}{Remarks}
\newtheorem*{remark*}{Remark}

\newenvironment{enumerate-alph}
{\begin{enumerate}
\addtolength{\itemsep}{5pt}
}
{\end{enumerate}}

\newenvironment{enumerate-num}
{\begin{enumerate}
\addtolength{\itemsep}{5pt}
}
{\end{enumerate}}

\newenvironment{enumerate-text}
{\begin{enumerate}
\addtolength{\itemsep}{5pt}
}
{\end{enumerate}}

\begin{document}

\title{Gabor unconditional bases and frames in $L^p(\R)$}

\author{Nir Lev}
\address{Department of Mathematics, Bar-Ilan University, Ramat-Gan 5290002, Israel}
\email{levnir@math.biu.ac.il}

\author{Anton Tselishchev}
\address{St. Petersburg Department of Steklov Mathematical Institute, Fontanka 27, St. Petersburg 191023, Russia}
\email{celis\_anton@pdmi.ras.ru}

\date{May 17, 2026}
\thanks{Research supported by ISF Grant No.\ 854/25 and the Foundation for the Advancement of Theoretical Physics and Mathematics ``BASIS''}

\begin{abstract}
We consider the following problem: given a set $\Lambda \subset \mathbb{R} \times \mathbb{R}$ and $p \neq 2$, does there exist a function $g \in L^p(\mathbb{R})$ such that the Gabor system $\{g(x-t) e^{2\pi i s x}\}$, $(t,s) \in \Lambda$, consisting of time--frequency shifts of $g$, forms an unconditional basis or unconditional Schauder frame in the space $L^p(\mathbb{R})$? We completely resolve this question for $p>2$; in particular, we characterize the sets $\Lambda$ such that an unconditional Schauder frame of this form exists. We also prove a Balian-Low type result, showing that the window function $g$ cannot enjoy mild continuity and decay conditions. For $1<p<2$, we prove that a Gabor system cannot form an unconditional basis or unconditional Schauder frame in $L^p(\mathbb{R})$ if the set $\Lambda$ satisfies a natural separation condition.
\end{abstract}

\maketitle

% =======================================

\section{Introduction}

\subsection{}

For a function $g$ on the real line $\R$ and a countable set $\Lambda\subset\R \times \R$, the \emph{Gabor system} generated by $g$ and $\Lambda$ is the system of functions 
\begin{equation}\label{eq:Gabordef}
\G(g;\Lambda)=\{ g(x-t)e^{2\pi i s x} : (t,s)\in\Lambda\}.
\end{equation}
The function $g$ is often called a \emph{window function}, and the elements of the system are called \emph{time--frequency shifts} of the function $g$.

Gabor systems are one of the main objects of study in time--frequency analysis. In particular, one of the most important questions is whether Gabor systems can form good coordinate systems in various function spaces.

Recall that a system of vectors $\{u_n\}$ in a separable Hilbert space $H$ 
is called a \emph{Riesz basis} if it can be obtained as the image of an orthonormal basis 
under a bounded and invertible linear operator. This condition implies that 
(i) every element $x\in H$ admits a unique series expansion
$x=\sum_{n}c_n u_n$, and this series converges
unconditionally, i.e.\ it converges for any rearrangement of its terms; and
(ii) there exist positive constants $A,B$ such that $A \le \|u_n\| \le B$ for every $n$.
Moreover, it is known that properties (i) and (ii) characterize the Riesz bases $\{u_n\}$ in $H$ 
(see e.g. \cite[Lemma 3.6.9]{Chr16}).

It is well known that Gabor systems can form Riesz bases (or even orthonormal bases) in the space $L^2(\R)$, the basic example being $\G(\1_{[0,1]};\Z\times\Z)$. On the other hand, the classical Balian--Low theorem states that if the function $g$ satisfies sufficient smoothness and decay conditions, then $\G(g; \Z \times \Z)$ cannot form a Riesz basis in the space $L^2(\R)$. There exist various versions of this result; in particular, if $\G(g; \Z\times \Z)$ is a Riesz basis in the space $L^2(\R)$ then $g$ cannot be simultaneously continuous and belong to the 
Wiener amalgam space $W(L^\infty, \ell^1)$, see \cite[Theorem 11.33]{Hei11}.

The Balian--Low theorem is one of the motivations to study more general coordinate systems formed by time--frequency shifts of a single function. A system of vectors $\{u_n\}$  in a separable Hilbert space $H$ 
is called a \emph{frame} (in the sense of  Duffin and Schaeffer \cite{DS52})
if there exist positive constants $A,B$  such that the inequalities
\begin{equation}\label{eq:frame_ineq}
A \|x\|^2 \le \sum_n |\langle x, u_n \rangle|^2\le B\|x\|^2
\end{equation}
hold for every $x\in H$. In this case there exists another frame
$\{v_n\}\subset H$ such that 
\begin{equation}\label{eq:frame_convergence}
x=\sum_{n} \dotprod{x}{v_n}u_n 
\end{equation}
for every $x \in H$,
and the series converges unconditionally
(see \cite[Section 4.7]{You01}).

We note that if
$\{u_n\}$  is a Riesz basis in $H$ then it is a frame, but in general
a frame need not be a Riesz basis and the
system $\{v_n\}$ satisfying  \eqref{eq:frame_convergence}
need not be unique.

If a system $\G(g; \Lambda)$ is a frame in the space $L^2(\R)$, then, naturally, it is called a Gabor frame. There are numerous results on Gabor frames. We do not attempt to survey the most recent results here and refer the reader to \cite[Chapter 11]{Hei11} and \cite[Chapters 11--13]{Chr16} for some classical background on this subject.

\subsection{}
There are several ways to define analogues of Riesz bases and frames in the Banach space setting, see e.g. \cite[Chapter 24]{Chr16}. If we are interested only in series expansion of vectors with respect to a coordinate system, but not in estimates similar in some sense to \eqref{eq:frame_ineq}, then the most natural notions are unconditional bases and unconditional Schauder frames.

A system $\{u_n\}_{n=1}^{\infty}$ in a separable Banach space $X$ is called a \emph{Schauder basis} if every $x\in X$ admits a unique series expansion $x = \sum_{n=1}^{\infty} c_nu_n$. If this series moreover converges unconditionally for every $x \in X$, then $\{u_n\}$ is called an \emph{unconditional basis}. The most well-known example of an unconditional basis in the space $L^p[0,1]$, $1 < p < \infty$, is the Haar basis (see \cite[Theorem 6.1.7]{AK16}). The system $\{u_n\}$ is called an \emph{unconditional basic sequence} if it is an unconditional basis for its closed linear span.

If  $X$ is a separable Banach space with dual space $X^*$, then a system of elements
$\{(u_n, u_n^*)\}_{n=1}^{\infty} \sbt X\times X^*$ is called a \emph{Schauder frame} if every $x\in X$ admits a series expansion given by
\begin{equation}
\label{eq:R1.1}
x=\sum_{n=1}^{\infty} u_n^*(x) u_n.
\end{equation}
If the series \eqref{eq:R1.1} 
converges unconditionally for every $x \in X$, then
$\{(u_n, u_n^*)\}$
is called an \emph{unconditional Schauder frame}.

Note that if $\{u_n\}$ is an
unconditional basis in $X$, then there exists a unique system 
of biorthogonal  coefficient functionals
$\{u_n^*\} \sbt X^*$
such that \eqref{eq:R1.1} holds for every $x \in X$
(see \cite[Section 1.1]{AK16}). In this sense, every 
unconditional basis is an unconditional Schauder frame.
However, in general the coefficient functionals
$\{u_n^*\}$ need not be unique
and they are not necessarily biorthogonal to $\{u_n\}$. 

\subsection{}
It is possible for a Gabor system to form a (not unconditional) Schauder basis in the space $L^p(\R)$, $p\neq 2$; for example, the system $\G(\1_{[0,1]}; \Z\times\Z)$ can be ordered to become a Schauder basis in the space $L^p(\R)$, $1 < p < \infty$. It is not known whether a Schauder basis in $L^p(\R)$, $1<p<\infty$, may be formed by only translates of a single function (see \cite[Problem 4.4]{OSSZ11}). However it is known that a system of translates can form a Schauder frame in $L^p(\R)$, see \cite[Section 4]{FPT21}, \cite{LT26a}, \cite{LT26b} for some recent results on this topic.

On the other hand, $L^p(\R)$ spaces for $p\neq 2$ are considered not the right spaces to search for unconditionally convergent Gabor expansions, see \cite[Section 24.5]{Chr16} or \cite{GH01}. But as we will see below, the situation is in fact more complicated, and the study of Gabor systems forming an unconditional basic sequence or an unconditional Schauder frame in $L^p(\R)$ spaces gives rise to several non-trivial results.

The present paper is devoted to the study of unconditional Gabor frames, bases and basic sequences in $L^p(\R)$ spaces for $p\neq 2$. We consider only values $p > 1$ since in the space $L^1(\R)$ 
there are no unconditional Schauder frames, see e.g. \cite[Section 4.3]{BC20}.

There are not many results showing that systems arising in time--frequency analysis cannot form unconditional bases in $L^p(\R)$ spaces, $p\neq 2$. We mention \cite[Theorem 2]{FGW92} which states that Wilson bases are not unconditional in $L^p(\R)$. We do not know any results of this kind for unconditional Schauder frames. See also \cite{GH01}, \cite{GL01} for some related results on Gabor expansions in $L^p(\R)$ spaces.

\subsection{}
Gabor systems are formed by translations and modulations of a single function. But there are several nontrivial questions already about unconditional Schauder bases and frames formed only by translates of a single function in $L^p(\R)$. It  is known that a system consisting of translates cannot form an unconditional basis in any of the spaces $L^p(\mathbb{R})$ \cite{OZ92}, \cite{OSSZ11}, \cite{FOSZ14}.

Some other nontrivial results about unconditional basic sequences and Schauder frames formed by translates of a single function can be found in \cite{OSSZ11}, \cite{FOSZ14}, \cite{BC20}, \cite{LT25a}. Specifically, in \cite{FOSZ14} it was shown that unconditional Schauder frames of translates exist in $L^p(\R)$ if $p > 2$ (in particular, this gives examples of uncon\-ditional Gabor frames in these spaces). On the other hand, we proved in \cite{LT25a} that for $p\le 2$ there are no unconditional Schauder frames of translates in $L^p(\R)$.

Another question closely related to the topic of the present paper is the problem about existence of unconditional Schauder frames of exponentials in the space $L^p[0,1]$ which was solved recently in \cite{LT25b}. While it has long been known that exponentials cannot form an unconditional basis in $L^p$ for $p\neq2$, the non-existence of unconditional Schauder frames of exponentials turned out to be a more subtle fact, and the study of this problem led to certain interesting results.

The results of the present paper, which we state below, are close in spirit to the results in the above mentioned papers. In the proofs we further develop the approaches from \cite{OSSZ11}, \cite{FOSZ14}, \cite{LT25a}, \cite{LT25b} and use also several new ideas from Fourier analysis and Banach space theory.

% =======================================

\section{Results}

\subsection{}

We begin the presentation of our results with the case $p > 2$.

\begin{thm}\label{thm:nobases_p>2}
If $p > 2$ then there does not exist any Gabor system
$\G(g;\Lambda)$ that forms an
unconditional basis in the space $L^p(\R)$.
\end{thm}

We emphasize the generality of this result: there are no a priori assumptions on the window function $g\in L^p(\R)$ or the countable time--frequency shift set $\Lambda\subset\R \times \R$. 

\thmref{thm:nobases_p>2} will be obtained as a consequence of a more general result: we will prove that if the Gabor system $\G(g;\Lambda)$ forms an unconditional basis for a complemented subspace $X \sbt L^p(\R)$, $p>2$, then the system is equivalent to a subsequence of the standard unit vector basis of the Banach space $(\ell^2 \oplus\ell^2 \oplus\cdots)_p$. In turn, this implies that $X \ne L^p(\R)$. We moreover present an example showing that the requirement that the subspace $X$ be complemented is crucial in this result (see \secref{sec:basesp>2}).

\subsection{}
Next, we turn to the problem of existence of Gabor systems $\G(g;\Lambda)$ that form an unconditional Schauder frame in the space $L^p(\R)$, $p>2$. The following result completely characterizes the time--frequency shift sets $\Lam$ for such Gabor systems.

\begin{thm}\label{thm:framesp>2}
\textup{(i)} Let $2 < p < \infty$ and $\Lambda\subset\R \times \R$ be a countable set which is not contained in any vertical strip $[-a, a]\times \R$. Then there exist a function $g\in L^p(\R)$ and a sequence $\{ \gtss \}$ in $(L^p(\R))^*$ such that the system
\begin{equation}
\label{eq:gaborschframe}
\{(\gtsp, \gtss) : (t,s)\in\Lambda\}, \quad \gtsp(x) = g(x-t)e^{2\pi i s x},
\end{equation}
forms an unconditional Schauder frame in the space $L^p(\R)$;

\textup{(ii)} To the contrary, if $ p \ge 2$ and $\Lambda\subset [-a,a]\times\R$, then no system \eqref{eq:gaborschframe} can  form an unconditional Schauder frame in $L^p(\R)$.
\end{thm}

In other words, an unconditional Gabor frame  in $L^p(\R)$, $2 < p < \infty$, with a given time--frequency shift set $\Lam$ exists if and only if the time shifts (i.e.\ the translates) in the set $\Lam$ form an unbounded sequence; at the same time $\Lam$ can be arbitrarily sparse. The proof of part (i) is done by a minor modification of the construction in \cite{FOSZ14} of unconditional Schauder frames of translates. 

The main novelty, thus, lies in part (ii), which is new even for $p=2$. In fact, by an application of Plancherel's theorem, our result implies that also if the set $\Lam$ is contained in a \emph{horizontal} strip $\R \times [-a,a]$, then no system \eqref{eq:gaborschframe} can  form an uncon\-ditional Schauder frame in $L^2(\R)$. 
This generalizes our recent result from \cite{LT25a} on the non-existence of unconditional Schauder frames of translates in the space $L^2(\R)$.

The case $p > 2$ is, however, significantly more complicated than the  $p=2$ case, mainly due to the absence of Plancherel's theorem in $L^p(\R)$ spaces. Our approach involves a certain substitute, namely, Rubio de Francia's Littlewood--Paley inequality for arbitrary intervals, which was proved in \cite{Rub85}.

\subsection{}
Once we know that unconditional Gabor frames do exist in the space $L^p(\R)$, $p > 2$, it is natural to ask whether the window function $g$ can be ``nice'', i.e.\ enjoy certain smoothness and decay properties. Our next theorem is a Balian--Low type result giving a negative answer in this direction.

Recall that the \emph{Wiener amalgam space} $W(L^\infty, \ell^1)$, that
we will shortly denote by $W$, consists of all measurable
functions $g$ on $\R$ satisfying
\begin{equation}
\label{eq:amalgam}
\|g\|_W =\sum_{k\in\Z} \|g|_{[k,k+1]}\|_{\infty} < \infty,
\end{equation}
see e.g.\ \cite[Section 11.4]{Hei11}.
It is easy to see that $W \subset L^p(\R)$, $1\le p\le\infty$.

We also recall that a set $\TTT\subset\R$ is called 
\emph{uniformly discrete} if the distance between 
any two distinct points of $\TTT$ is bounded from 
below by a positive constant.

\begin{thm}
\label{thm:nicegen}
Let $p \ne 2 $, and suppose that $g$ is a continuous function in $W$. 
Then for any uniformly discrete set $\TTT\subset \R$ and any 
countable set $\Lam \sbt \TTT\times \R$, 
no system \eqref{eq:gaborschframe} can form 
an unconditional Schauder frame in $L^p(\R)$.
\end{thm}

In fact, the continuity assumption can be relaxed:
our proof works for any function $g\in W$
which can be uniformly approximated by compactly supported
step functions. Also, the result remains valid in the more general case where 
$\TTT \sbt \R$ is a set of bounded density, i.e.\ a finite union of uniformly discrete sets.

The reader may compare \thmref{thm:nicegen} to the Balian--Low type result 
in \cite[Theorem 11.33]{Hei11} mentioned above. Another result worth mentioning here is that if $g\in (L^p \cap L^1)(\R)$, then no system of uniformly discrete translates of $g$ may form a Schauder frame (even not unconditional) in $L^p(\R)$, see \cite[Proposition 2.2]{OSSZ11}.

\subsection{}
We now turn to the case $1 < p < 2$. For these values of $p$ the situation changes, 
and our result states that under a uniform discreteness assumption, no unconditional 
Gabor bases or frames exist in the space $L^p(\R)$.

\begin{thm}
\label{thm:gabframesp<2}
Let $1<p<2$, and assume that
$\Lambda\subset \TTT\times \R$ for some uniformly discrete set $\TTT\subset\R$.
Then no system \eqref{eq:gaborschframe} can form an unconditional Schauder frame in $L^p(\R)$.
In particular, there is no unconditional basis of the form $\G(g;\Lambda)$ in $L^p(\R)$.
\end{thm}

Moreover, we prove a result about unconditional Gabor basic sequences:
 if the Gabor system $\G(g;\Lambda)$ forms an unconditional basic sequence in $L^p(\R)$, 
 $1 \le p < 2$, and if
 $\Lambda\subset \TTT\times \R$ for some uniformly discrete set $\TTT\subset\R$,
  then the system is equivalent to a subsequence of the standard unit vector basis of $(\ell^2 \oplus\ell^2 \oplus\cdots)_p$. As a consequence, the system $\G(g;\Lambda)$ cannot span $L^p(\R)$. Note that in this result (contrary to the $p>2$ case) we do not assume that the system $\G(g;\Lambda)$ spans a complemented subspace of $L^p(\R)$.

The results are true more generally in the case where
$\TTT \sbt \R$ is a finite union of uniformly discrete sets.
On the other hand, we will present an example showing that 
our result about unconditional Gabor basic sequences
in $L^p(\R)$, $1 \le p<2$,
becomes false if the uniform discreteness assumption
 is dropped (see \secref{sec:basesp<2}).

We suspect that no unconditional Gabor bases or frames exist in $L^p(\R)$, $1 < p <2$, without imposing any a priori assumptions on the countable set $\Lambda \sbt \R \times \R$ or the
 function $g \in L^p(\R)$; however we leave this problem open.

\subsection{}
In the remainder of the paper, we prove the theorems stated above and
 provide some additional remarks about the results. 
 In \secref{sec:prelim}, we review some preliminary 
basic facts from Banach space theory and Fourier analysis, 
which are used throughout the paper. 
In Sections \ref{sec:basesp<2} and \ref{sec:framesp<2} we 
prove our results on unconditional Gabor basic sequences and 
frames in $L^p(\R)$ spaces, $1 < p <2$.
In Sections \ref{sec:basesp>2}, \ref{sec:framesp>2} 
and \ref{sec:nicewind}, the results concerning  the case $p>2$ are presented.

\emph{Remark}. The results extend, with essentially the same proofs, to higher dimensions,
i.e., they hold for Gabor systems in $L^p(\R^d)$ for any dimension $d$.
We present the results in dimension one merely in order to simplify the exposition.

% =======================================

\section{Preliminaries}
\label{sec:prelim}

In this section, we review some preliminary 
basic facts from Banach space theory and Fourier analysis, 
that will be used throughout the paper. 

\subsection{}
\emph{Notation}. 
For $t\in\R$ we use $\ts_t g$ to denote translates of a function $g\in L^p(\R)$, i.e. 
\begin{equation}
(\ts_t g)(x)=g(x-t), \; x \in \R.
\end{equation}
For $s \in \R$ we use $e_s$ to denote the exponential function with frequency $s$,
\begin{equation}
e_s(x) = e^{2\pi i s x}, \;  x \in \R.
\end{equation}
Thus, time-frequency shifts of $g$ will be denoted as 
\begin{equation}
(e_s \ts_t g)(x) =  g(x-t) e^{2 \pi i s x} 
\end{equation}
where $(t,s) \in \R \times \R$.

We will write $A\lesssim B$ meaning that $A\le CB$ holds for some positive constant $C$ (it will   be clear from the context on what parameters the constant $C$ can depend). Also, $A\asymp B$ means that we have both $A\lesssim B$ and $B\lesssim A$.

\subsection{}
We start with a basic property of unconditional Schauder frames, that can be proved
using the uniform boundedness principle in the same way 
as a similar statement for unconditional Schauder bases (see e.g.\ \cite[Proposition 3.1.3]{AK16}).

\begin{lemma}
\label{lem:unc_constant}
Let $\{(u_j, u_j^*)\}_{j=1}^\infty$ be an unconditional Schauder frame in 
a Banach space $X$.  Then there exists a constant $K$ such that for any 
$x\in X$ and for any
sequence of scalars $\{\theta_j\}$ satisfying  $|\theta_j|\le 1$, we have
\begin{equation}
\label{eqtheta}
\Big\| \sum_{j=1}^\infty \theta_j u_j^*(x) u_j \Big\| \leq K \|x\|
\end{equation}
and the series in \eqref{eqtheta} converges unconditionally.
\end{lemma}

In this case we say that  
$\{(u_j, u_j^*)\}_{j=1}^\infty$ 
is a \emph{$K$-unconditional} Schauder frame. In a similar way we define $K$-unconditional bases and basic sequences.

If $\{u_j\}_{j=1}^\infty$ is a $K$-unconditional basis in the space $X$, then it follows from uniqueness of the expansion that if $|\theta_j| = 1$ for all $j$ then also a reverse estimate holds,
\begin{equation}
\|x\|\le K \Big\| \sum_{j=1}^\infty \theta_j u_j^*(x) u_j \Big\|.
\end{equation}

\subsection{}
Following \cite{FOSZ14}, we introduce the notion of approximate Schauder frame.

\begin{definition}
Let $X$ be a separable Banach space with dual space $X^*$. 
A sequence of elements $\{(u_n, u_n^*)\}_{n=1}^\infty \subset X\times X^*$ is called an \emph{approximate Schauder frame} for $X$ if for any $x\in X$ the series
\begin{equation}\label{eq:approx_frame}
Sx = \sum_{n=1}^\infty u_n^*(x) u_n
\end{equation}
converges in $X$ and defines a bounded and invertible linear operator $S: X \to X$. If additionally the series \eqref{eq:approx_frame} converges unconditionally for every $x\in X$, then the system 
$\{(u_n, u_n^*)\}_{n=1}^\infty$ is called an \emph{unconditional approximate Schauder frame}.
\end{definition} 

The main property of an approximate Schauder frame is that it can be turned into a Schauder frame by a change of coordinate functionals. We formulate this fact in the following lemma.

\begin{lemma}\label{lem:approxframe}
If $\{(u_n, u_n^*)\}_{n=1}^\infty$ forms an
unconditional approximate Schauder frame for $X$, 
then $\{(u_n, (S^{-1})^* u_n^*)\}_{n=1}^\infty$
forms an unconditional Schauder frame for $X$.
\end{lemma} 
The proof is straightforward, see \cite[Lemma 3.1]{FOSZ14}.

\subsection{}
We will constantly use Khintchine's inequality. It will be convenient for us to use the language and notation of probability theory.  
If $(\Omega, \PP)$ is a probability space, then the \emph{expectation} (or \emph{mean})
of a function $f \in L^1(\Omega, \PP)$ is defined by
$\E f = \int_{\Omega} f d\PP$.

A \emph{Rademacher sequence} is a sequence of i.i.d.\ random variables $\{\eps_n\}_{n=1}^\infty$ defined on some probability space $(\Omega, \PP)$ such that $\PP(\eps_n = 1) = \PP(\eps_n = -1) = 1/2$ for every $n$.

The classical Khintchine's inequality states that for any $1\le p < \infty$ there exist positive constants $A_p$ and $B_p$ such that for every $N$ and any sequence of scalars $\{a_n\}$,
\begin{equation}\label{eq:Khintch}
A_p \Big( \sum_{n=1}^N |a_n|^2 \Big)^{1/2}\le \Big(\E \Big| \sum_{n=1}^N a_n\eps_n \Big|^p\Big)^{1/p} \le B_p\Big( \sum_{n=1}^N |a_n|^2 \Big)^{1/2}.
\end{equation}
Since $\{\eps_n\}$ is an orthonormal sequence in $L^2(\Omega, \PP)$, we can put $A_p = 1$ for $p\ge 2$ and $B_p = 1$ for $p\le 2$. We refer the reader to \cite[Section 6.2]{AK16} for more information about Khintchine's inequality and averaging in Banach spaces. In particular, we will use the following corollary of Khintchine's inequality, see \cite[Theorem 6.2.13]{AK16}.

\begin{lem}\label{lem:squarefunc}
Let $1\le p < \infty$. For any finite sequence of functions $\{f_j\}_{j=1}^N$ in $L^p(\R)$,
\begin{equation}
A_p \Big\| \Big(\sum_{j=1}^N |f_j|^2\Big)^{1/2} \Big\|_p\le \Big( \E \Big\| \sum_{j=1}^N \eps_j f_j \Big\|_p^p \Big)^{1/p}\le B_p \Big\| \Big(\sum_{j=1}^N |f_j|^2\Big)^{1/2} \Big\|_p,
\end{equation}
where $A_p$ and $B_p$ are the constants from Khintchine's inequality \eqref{eq:Khintch}.
\end{lem} 

Another useful consequence of Khintchine's inequality is the fact that the space $L^p(\R)$, $1\le p\le 2$, has cotype 2 which we formulate in the following lemma; see \cite[Theorem 6.2.14]{AK16} for the proof.

\begin{lemma}
\label{lem:cotype}
Let $1\le p \le 2$. There exists a positive constant $C_p$ such that for any finite 
sequence of functions $f_1, f_2, \ldots, f_N \in L^p(\R)$ we have 
\begin{equation}
\label{eq:cotype}
\Big( \sum_{j=1}^N \|f_j\|_p^2 \Big)^{1/2}\leq C_p \, \E \Big\|\sum_{j=1}^N \eps_j f_j\Big\|_p.
\end{equation}
\end{lemma}

On the other hand, if $2\le p < \infty$, then the space
$L^p(\R)$ has type 2; this fact is formulated as follows.

\begin{lemma}
\label{lem:type}
Let $2\le p < \infty$. There exists a positive constant $C_p$ such that for any finite 
sequence of functions $f_1, f_2, \ldots, f_N \in L^p(\R)$ we have 
\begin{equation}
\label{eq:type}
\E \Big\|\sum_{j=1}^N \eps_j f_j\Big\|_p\le C_p\Big( \sum_{j=1}^N \|f_j\|_p^2 \Big)^{1/2}.
\end{equation}
\end{lemma}
Again, we refer the reader to \cite[Theorem 6.2.14]{AK16} for the proof.

\subsection{}

The following inequality can be proved by an application of Minkowski's integral inequality with exponent $p/2$ (see e.g.\ \cite[p.\ 156]{AK16}).

\begin{lem}\label{lem:mink}
If $2\le p < \infty$, then $\Big\| \Big( \sum_j |f_j|^2\Big)^{1/2} \Big\|_p\le \Big( \sum_j \|f_j\|_p^2 \Big)^{1/2}$.
\end{lem}

If $1 \le p \le 2$, then the reverse inequality holds, but we will not use this fact.

It is sometimes convenient to denote the quantity on the left-hand side of the above inequality as $\|(f_j)_j\|_{L^p(\ell^2)}$. This norm on sequences of functions defines the Banach space $L^p(\ell^2)$. If we want to emphasize the domain of definition of the functions $\{f_j\}$, say, they are defined on $\Omega$, we will use the notation $\|(f_j)_j\|_{L^p(\Omega;  \ell^2)}$.

\subsection{}
We will need the following result.

\begin{lem}\label{lem:JO.EXT}
Let $\{f_n\}$ be a $K$-unconditional basic sequence in $L^p(\R)$, $1\le p\le 2$,
such that  $C = \sup_n \|f_n\|_p$ is finite.
Assume that there exist measurable sets 
$E_n\subset\R$ and constants $M$ and $\del>0$, such that 
$\sum_n \1_{E_n}\le M$ a.e., and 
$\|f_n|_{E_n}\|_p\ge\delta$ for every $n$. 
Then $\{f_n\}$ is equivalent to the standard unit vector basis of $\ell^p$, namely,
\begin{equation}\label{eq:2.6}
\Big\| \sum_n a_n f_n \Big\|_p \asymp \Big( \sum_n |a_n|^p \Big)^{1/p}
\end{equation}
for any sequence of scalars
$\{a_n\}$ with only finitely many nonzero terms.
Moreover, the implied constants in \eqref{eq:2.6} depend only on $p$, $K$, $C$, $M$ and $\delta$, and otherwise not on the specific sequence of functions $\{f_n\}$. 
\end{lem}

This result is a straightforward generalization of 
\cite[Section 3, Lemma 2]{JO74}, and can be
proved in essentially the same way.
In \cite{JO74} the lemma was stated such that the sets
$E_n$ are required to be pairwise disjoint.
The condition $\sum_n \1_{E_n}\le M$ a.e.\ 
can be interpreted as a relaxation of this requirement:
 it means that almost every point of 
the real line can lie in no more than $M$ of the sets $E_n$.

\subsection{}

If we are given a (finite or infinite) sequence of Banach spaces $\{X_n\}$, then its direct $\ell^p$-sum is defined as
\begin{equation}
\label{eq:DIR.LP}
\Big(\bigoplus_n X_n\Big)_p = \Big\{(x_n) : x_n\in X_n \; \text{and} \; \|(x_n)\|_p := \Big(\sum_n \|x_n\|_{X_n}^p\Big)^{1/p} < \infty\Big\}.
\end{equation}
See \cite[II.B.21]{Woj91} for some information on direct sums of Banach spaces.

We will be mostly interested in the case where $X_n = H_n$ is a separable Hilbert space  (either finite or infinite dimensional) for every $n$. If all the spaces $H_n$ are infinite dimensional and there are infinitely many of them, then the direct sum \eqref{eq:DIR.LP} is the Banach space 
$\ell^p(\ell^2)=(\ell^2\oplus\ell^2\oplus\cdots)_p$; otherwise, it is a closed subspace of $\ell^p(\ell^2)$. In particular, for $p\neq 2$ it follows that this space is not isomorphic to $L^p$ due to the following classical result.

\begin{lem}\label{lem:noniso}
If $1 \le p < \infty$, $p\neq 2$, then the space $L^p$ is not isomorphic to a closed subspace of $\ell^p(\ell^2)$.
\end{lem}

This fact can be found in \cite[Example 8.2]{LP68} for $1<p < 2$, and in \cite[Theorem 6.1]{LP71} for $p > 2$. The case $p=1$ may be deduced from the fact that the space $L^1$ is not isomorphic to a closed subspace of any space with
an unconditional basis, see \cite[Section II.D, Theorem 10]{Woj91}).

\subsection{}
It is well known that if $\SSS \sbt \R$ is a uniformly discrete set, then there 
exists a constant $C = C(S) > 0$ such that for any interval $I\subset\R$ of unit length, and for arbitrary scalars $\{a_s\}_{s\in\SSS}$ such that only finitely many of them are nonzero, the estimate 
\begin{equation}
\Big\| \sum_{s\in \SSS} a_se_s \Big\|_{L^2(I)}\le C \Big( \sum_{s\in \SSS} |a_s|^2 \Big)^{1/2}
\end{equation}
holds. In order to deal with non-uniformly discrete sets $\SSS$ we will need a generalization of this result, see \cite[Section 5, p.14]{LT25b} for the proof. 

\begin{lem}
\label{lem:bess}
For any $\delta > 0$ there exists a constant $C(\delta) > 0$ with the following property. Given an arbitrary countable set $\SSS\subset\R$, denote $\SSS_l = \SSS\cap [l\delta, (l+1)\delta)$, $l\in\Z$. Then for any interval $I\subset\R$ of unit length, and for arbitrary scalars
$\{a_s\}_{s\in \SSS}$ we have
\begin{equation}
\label{eq:as.es}
\Big\| \sum_{s\in \SSS} a_se_s \Big\|_{L^2(I)}\le C(\delta)  \Big(\sum_{l\in\Z} \Big( \sum_{s\in \SSS_l}|a_s| \Big)^2\Big)^{1/2}
\end{equation}
provided that
the quantity on the right hand side is finite. In this case,
the series on the left hand side converges unconditionally in $L^2(I)$.
\end{lem}

\subsection{}
A sequence of positive real numbers $\{s_n\}_{n=1}^{\infty}$ is called \emph{(Hadamard) lacunary} if there is a real number $\lambda>1$ such that $s_{n+1}/s_n \ge \lambda$ for all $n$. It is well known that in this case, the exponential system $\{e_{s_n}\}_{n=1}^{\infty}$ behaves 
similar to a Rademacher sequence. Specifically, we will use the following fact,
see \cite[Chapter V, Theorem 8.20]{Zyg59}, i.e.\ a version of Khintchine's inequality for lacunary trigonometric series.

\begin{lem}
\label{lem:lacun}
Let $\{s_n\}_{n=1}^{\infty}$ be 
a sequence of positive integers such  that $s_{n+1}/s_n \ge \lambda > 1$ for all $n$. Then for any $1 \le p < \infty$, 
every $N$ and  any sequence of scalars $\{a_n\}$,
\begin{equation}
\label{eq:kh:lac}
A_{p,\lam} \Big( \sum_{n=1}^N |a_n|^2 \Big)^{1/2}\le \Big(
\int_{0}^{1}  \Big| \sum_{n=1}^N  a_n e_{s_n}(x)  \Big|^p dx \Big)^{1/p} \le B_{p, \lam} \Big( \sum_{n=1}^N |a_n|^2 \Big)^{1/2}
\end{equation}
for some positive constants $A_{p,\lam}$ and $B_{p, \lam}$ which depend on $p$ and $\lam$ only.
\end{lem}

% =======================================

\section{The case $1 <p < 2$: Unconditional Gabor basic sequences}
\label{sec:basesp<2}

In this section we prove that a Gabor system
$\G(g;\Lambda)$ such that 
$\Lambda\subset \TTT\times \R$ for some uniformly discrete set $\TTT\subset\R$,
cannot form an unconditional basis 
 in the space $L^p(\R)$, $1<p<2$.
We obtain this as a consequence of a more
general result, about unconditional Gabor basic sequences in
$L^p(\R)$.

\subsection{}
We start by stating our result about unconditional Gabor basic sequences.
 To this end, observe that
any countable set $\Lambda\subset \R\times\R$ may be represented in the form
\begin{equation}\label{eq:Lambda}
\Lambda=\bigcup_{t\in \TTT} \{t\}\times \SSS_t,
\end{equation}
where $\TTT \sbt \R$ is a finite or countable set, and each $\SSS_t\subset\R$ is a finite or countable nonempty set.

We denote by $\ell^2_{\SSS_t}$ the space $\ell^2$ of sequences indexed by the set $\SSS_t$.

\begin{thm}
\label{thm:uncseq_p<2}
Let $1\le p \le 2$ and let $\Lambda$ be a set of the form \eqref{eq:Lambda}, where $\TTT \sbt \R$ is a uniformly discrete set. Suppose that for some $g\in L^p(\R)$, the system $\G(g;\Lambda)$ forms an unconditional basic sequence in $L^p(\R)$. Then the system $\G(g;\Lambda)$  is equivalent to the standard unit vector basis of the space $(\bigoplus_{t\in \TTT} \ell^2_{\SSS_t})_p$.
\end{thm}

The elements of the standard unit vector basis of the space $(\bigoplus_{t\in \TTT} \ell^2_{\SSS_t})_p$ are naturally indexed by points of the set $\{(t,s): t\in \TTT, s\in \SSS_t\}$, that is, by the points of $\Lambda$. The equivalence stated in the theorem means that the linear operator which maps $e_s \ts_tg$, $(t,s)\in \Lambda$, to the corresponding element of the unit vector basis of $(\bigoplus_{t\in \TTT} \ell^2_{\SSS_t})_p$ is an isomorphism between the subspace $X = \cspan \, \G(g;\Lambda)$ and $(\bigoplus_{t\in \TTT} \ell^2_{\SSS_t})_p$.
In particular, for $1 \le p < 2$ we obtain that
$X \neq L^p(\R)$ by \lemref{lem:noniso}; hence,
the system $\G(g;\Lambda)$ cannot form
an unconditional basis  in the space $L^p(\R)$.

As an example, let $\TTT=\Z$ and $\SSS = \{1,2,4,8,16,\dots\}$. Then using \lemref{lem:lacun} it is easy to verify that the system $\G(\1_{[0,1]};\TTT\times \SSS)$ forms an unconditional basic sequence in $L^p(\R)$, and
its closed span is isomorphic to $\ell^p(\ell^2)$. \thmref{thm:uncseq_p<2} states that, in a sense, an unconditional basic sequence $\G(g;\Lambda)$ can never span a bigger subspace of $L^p(\R)$.

\thmref{thm:uncseq_p<2} is also true (with the same proof) in the more general case where $\TTT$ is a set of bounded density, i.e.\ a finite union of uniformly discrete sets.

On the other hand, we will present an example showing that 
\thmref{thm:uncseq_p<2} becomes false 
if the uniform discreteness assumption is dropped
(see \thmref{thm:u.b.e.p<2} below).

Finally, we point out that the proof below actually works for arbitrary uni\-modular functions $h_s$ in place of complex exponentials $e_s$, or, even more generally, for any func\-tions $h_s$ such that $A \le |h_s(x)|\le B$ a.e.\ for some positive constants $A,B$.

\subsection{}
We now turn to prove \thmref{thm:uncseq_p<2}. If $\{e_s \ts_t g: (t,s)\in\Lambda\}$ is a $K$-unconditional basic sequence, then for any choice of signs $\theta_{ts}=\pm 1$ and any complex scalars $a_{ts}$ such that only finitely many of them are nonzero, we have
\begin{equation}\label{eq:3.2}
\frac{1}{K}\Big\|\sum_{(t,s)\in\Lambda}  a_{ts} e_s \ts_t g\Big\|_p\le \Big\|\sum_{(t,s)\in\Lambda} \theta_{ts} a_{ts} e_s \ts_t g\Big\|_p\le K 	\Big\|\sum_{(t,s)\in\Lambda}  a_{ts} e_s \ts_t g\Big\|_p.
\end{equation}
Let $\{ \eps_{ts} \}$ be a Rademacher sequence, then by \lemref{lem:squarefunc} we get
\begin{gather}
A_p\Big\| \Big( \sum_{(t,s)\in \Lambda}  |a_{ts}|^2  |e_s \ts_t g|^2 \Big)^{1/2} \Big\|_p \le \Big(\E \Big\|\sum_{(t,s)\in\Lambda} \eps_{ts} a_{ts} e_s \ts_t g\Big\|^p_p\Big)^{1/p} \\
\le B_p \Big\| \Big( \sum_{(t,s)\in \Lambda}  |a_{ts}|^2  |e_s \ts_t g|^2 \Big)^{1/2} \Big\|_p.
\end{gather}
If we put the random signs $\eps_{ts}$ instead of $\theta_{ts}$ into the inequalities \eqref{eq:3.2} and apply the above estimate, we conclude that 
\begin{equation}
\Big\|\sum_{(t,s)\in\Lambda}  a_{ts} e_s \ts_t g\Big\|_p \asymp \Big\| \Big( \sum_{(t,s)\in \Lambda}  |a_{ts}|^2  |e_s \ts_t g|^2 \Big)^{1/2} \Big\|_p.
\end{equation}
Now for each $t\in \TTT$ we choose and fix some element $s_t\in \SSS_t$. Obviously, for every $s\in \SSS_t$ we have the pointwise equality $|e_s \ts_tg| = |e_{s_t} \ts_tg|$. Therefore, 
\begin{equation}
\Big\| \Big( \sum_{(t,s)\in \Lambda}  |a_{ts}|^2  |e_s \ts_t g|^2 \Big)^{1/2} \Big\|_p = \Big\|\Big( \sum_{t\in \TTT} \Big( \sum_{s\in \SSS_t} |a_{ts}|^2 \Big)  |e_{s_t} \ts_t g|^2 \Big)^{1/2}\Big\|_p.
\end{equation}
The system $\{e_{s_t } \tau_t g\}_{t\in \TTT}$ is a subsystem of $\G(g;\Lambda)$; hence, it is also a $K$-unconditional basic sequence. Similarly to the above, we conclude that for arbitrary complex scalars $\{b_t\}$ with only finitely many terms, we have
\begin{equation}
\Big\| \sum_{t\in \TTT} b_t e_{s_t} \ts_t g \Big\|_p\asymp \Big\| \Big(\sum_{t\in \TTT} |b_t|^2  |e_{s_t} \ts_t g|^2 \Big)^{1/2} \Big\|_p.
\end{equation}
In particular, putting
\begin{equation}
b_t =  \Big( \sum_{s\in \SSS_t} |a_{ts}|^2 \Big)^{1/2}
\end{equation}
we get that 
\begin{equation}
\Big\|\Big( \sum_{t\in \TTT} \Big( \sum_{s\in \SSS_t} |a_{ts}|^2 \Big)  |e_{s_t} \ts_t g|^2 \Big)^{1/2}\Big\|_p \asymp \Big\| \sum_{t\in \TTT} \Big( \sum_{s\in \SSS}|a_{ts}|^2 \Big)^{1/2}  e_{s_t} \ts_t g \Big\|_p.
\end{equation}
If $\TTT$ is a uniformly discrete set, or more generally, a finite union of uniformly discrete sets, the system $\{e_{s_t} \ts_t g\}_{t\in \TTT}$ satisfies the conditions of \lemref{lem:JO.EXT}. Indeed, it is enough to consider an interval $I\subset\R$ of unit length such that $\del = \|g|_I\|_p > 0$, and put $E_t = I+t$ for $t\in \TTT$. So an application of \lemref{lem:JO.EXT} yields
\begin{equation}
\Big\| \sum_{t\in \TTT} \Big( \sum_{s\in \SSS}|a_{ts}|^2 \Big)^{1/2}  e_{s_t} \ts_t g \Big\|_p \asymp \Big( \sum_{t\in \TTT} \Big( \sum_{s\in \SSS_t}|a_{ts}|^2 \Big)^{p/2} \Big)^{1/p}.
\end{equation}
Summing up, we proved that
\begin{equation}
\Big\|\sum_{(t,s)\in\Lambda}  a_{ts} e_s \ts_t g\Big\|_p\asymp \Big( \sum_{t\in \TTT} \Big( \sum_{s\in \SSS_t}|a_{ts}|^2 \Big)^{p/2} \Big)^{1/p}, 
\end{equation}
which establishes the assertion of \thmref{thm:uncseq_p<2}.
\qed

\subsection{Example}
We now present an example showing that \thmref{thm:uncseq_p<2} 
becomes false if we drop the assumption that
$\Lambda\subset \TTT\times \R$ for some uniformly discrete set $\TTT\subset\R$.

In fact, in the example below, the set $\Lambda$ \emph{accumulates near a vertical line}. It is reasonable to assume that if the translates are close to one another, then they should ``behave in almost the same way'', and in this case, \thmref{thm:uncseq_p<2} suggests that the unconditional Gabor basic sequence should be equivalent to the standard unit vector basis of $\ell^2$. However, the example shows that this intuition fails for $1 \le p < 2$.

\begin{thm}
\label{thm:u.b.e.p<2}
Given $p$, $1 \le p < 2$, there is a Gabor system $\G(g;\Lambda)$ that forms an unconditional basic sequence in the space $L^p(\R)$, but the system is not equivalent to any sub\-sequence of the standard unit vector basis of $\ell^p(\ell^2)$.
\end{thm}

\begin{proof}
For arbitrary complex scalars $\{c_k\}_{k=1}^\infty$ such that $\sum_{k=1}^\infty |c_k|^p =1$, define
\begin{equation}
g = \sum_{k=1}^\infty c_k  2^{k/p} \, \1_{[k, k+2^{-k}]}.
\end{equation}
Then $\|g\|_p^p = \sum_{k=1}^\infty |c_k|^p = 1$.
We consider the Gabor system $\G(g;\Lambda)$ with
\begin{equation}\label{eq:example:set}
\Lam = \{ (2^{-j}, 2^j) : j=1,2,3,\dots \}.
\end{equation}

If  $a=\{a_j\}_{j=1}^\infty$ is a sequence of scalars with only finitely many nonzero terms, denote
\begin{equation}
\Phi_a = \sum_{j=1}^\infty a_j e_{2^j} \ts_{2^{-j}} g.
\end{equation}
We will obtain an expression for the norm $\|\Phi_a|_{[k,k+1]}\|_p$ on each interval $[k,k+1]$, $k\in\Z$. Since $\Phi_a$ is supported on $[1,  \infty)$, it suffices to consider only $k \ge 1$.

We observe that

1. For $1 \le l \le k-1$, on $[k+2^{-l}, k+2^{-l}+2^{-k}]$ we have $\Phi_a = c_k 2^{k/p} a_l e_{2^l}$;

2. For $l \ge k$, on $[ k+2^{-k}+2^{-l-1}, k+2^{-k}+2^{-l}]$ we have $\Phi_a = c_k 2^{k/p} \sum_{j=k}^{l} a_j e_{2^j}$;

3. For $l \ge k$, on $[ k+2^{-l-1}, k+2^{-l}]$ we have $\Phi_a = c_k 2^{k/p} \sum_{j>l} a_j e_{2^j} $.

These intervals have disjoint interiors and
they cover the support of $\Phi_a$ inside the interval $[k, k+1]$. Hence, $\|\Phi_a|_{[k,k+1]}\|_p^p = (1)+(2)+(3)$,
where the last three terms denote respectively the contributions from intervals of the above three types. 

It is obvious that 
$(1) = |c_k|^p \sum_{j=1}^{k-1} |a_j|^p$.

By an application of \lemref{lem:lacun}, we also get that
$(3) \asymp |c_k|^p (\sum_{j>k} |a_j|^2)^{p/2}$.

Next, observe that $(2) \gtrsim |c_k a_k|^p$. In order to estimate $(2)$ from above, we use the pointwise estimate
$ | \sum_{j=k}^{l} a_j e_{2^j} |^p \le (l-k+1)^p \max_{k \le j \le l}  |a_j|^p $, which after integration gives us that
$(2) \lesssim |c_k|^p \sup_{j \ge k} |a_j|^p$.

Combining the above three estimates, we conclude that
\begin{equation}
\label{eq:phi:k}
\|\Phi_a|_{[k,k+1]}\|_p^p \asymp
|c_k|^p \Big\{
\sum_{j=1}^{k} |a_j|^p + 
(\sum_{j>k} |a_j|^2)^{p/2} \Big\}, \quad k=1,2,\dots,
\end{equation}
with the implied constants depending only on $p$. 

Now put $w_j = \sum_{k=j}^\infty |c_k|^p$, then it follows from the estimate \eqref{eq:phi:k} that
\begin{equation}
\label{eq:phi:p}
\|\Phi_a\|_p \asymp \Big( \sum_{j=1}^\infty |a_j|^p w_j \Big)^{1/p} + \Big(\sum_{j=1}^\infty |a_j|^2\Big)^{1/2},
\end{equation}
with the implied constants now depending on $p$ and the sequence $\{c_k\}_{k=1}^{\infty}$ only.

In particular, \eqref{eq:phi:p} implies (e.g.\ using \cite[Propositions 1.1.9 and 3.1.3]{AK16}) that the system  $\G(g;\Lambda)$ forms an unconditional basic sequence in $L^p(\R)$.

We now show that, for a suitable choice of the sequence $\{c_k\}_{k=1}^{\infty}$, the system $\G(g;\Lambda)$ is not equivalent to any subsequence of the standard unit vector basis of $\ell^p(\ell^2)$. We observe that $\{c_k\}_{k=1}^{\infty}$ may be chosen so that $\{w_j\}_{j=1}^{\infty}$ is an arbitrary non-increasing sequence tending to zero, $w_1 = 1$. In particular,
$w_j$ may tend to zero arbitrarily slowly. Hence, if $1 \le p < 2$, we may choose $\{c_k\}_{k=1}^{\infty}$ so that $\sum_{j=1}^{n} w_j \gg n^{p/2}$ as $n \to \infty$.

Assume now that the system $\G(g;\Lambda)$ is equivalent to some subsequence of the standard unit vector basis of $\ell^p(\ell^2)$. We will show that this leads to a contradiction.

The assumption means that the set of positive integers $\{1,2,3,\dots\}$ may be partitioned into finitely or countably many nonempty disjoint sets $\{J_s\}$ such that
\begin{equation}
\label{eq:phi:c}
\|\Phi_a\|_p \asymp \Big( \sum_{s} \Big(\sum_{j \in J_s} |a_j|^2\Big)^{p/2} \Big)^{1/p}.
\end{equation}

If there are only finitely many sets $J_s$, then the right hand side of \eqref{eq:phi:c} is equivalent to the norm $\|a\|_{\ell^2}$. But then, if we choose the sequence
$a=(1,1,\dots,1,0,0,\dots)$ such that the number of $1$'s is $n$, then $\|a\|_{\ell^2} = n^{1/2}$, while
the right hand side of \eqref{eq:phi:p}
becomes $ ( \sum_{j=1}^n w_j )^{1/p} + n^{1/2} \gg n^{1/2}$ due to our choice of $\{c_k\}_{k=1}^{\infty}$. 
This contradicts \eqref{eq:phi:c}, so it is not possible that there are only finitely many sets $J_s$.

Hence, there must be infinitely many sets $J_s$. Suppose now that $a = \{a_j\}_{j=1}^{\infty}$ is a sequence whose support contains at most one element from each set $J_s$. In this case, the right hand side of \eqref{eq:phi:c} is equal to the norm $\|a\|_{\ell^p}$. 
If $a_j = 0$ or $1$ for each $j$, then
we have $\|a\|_{\ell^p} = n^{1/p}$, where $n$ is the size of the support of $a$.
But if, in addition, the support of $a$ does not intersect $\{1, 2, \dots, N-1\}$ for some $N$, then (recalling that the sequence $\{w_j\}_{j=1}^{\infty}$ is non-increasing) the right hand side of \eqref{eq:phi:p} does not exceed $ (n \cdot w_{N} )^{1/p} +  n^{1/2}$. Since $w_N \to 0$ as $N \to + \infty$, given $\eps>0$ we can therefore choose $n$ and $N$ such that the right hand side of \eqref{eq:phi:p} does not exceed $\eps \|a\|_{\ell^p}$, and we thus arrive at the desired contradiction.

This shows that indeed, the system $\G(g;\Lambda)$ is not equivalent to any subsequence of the standard unit vector basis of $\ell^p(\ell^2)$, and concludes the proof of \thmref{thm:u.b.e.p<2}.
\end{proof}

% =======================================

\section{The case $1 < p < 2$: Unconditional Gabor frames}
\label{sec:framesp<2}

Next we prove \thmref{thm:gabframesp<2}. We start with the following remark: we proved recently \cite{LT25b} that in the space $L^p[0,1]$, $1 < p < \infty$, there exists an unconditional Schauder frame consisting of \emph{unimodular functions} $\{h_s\}$, $s \in \Z$,
i.e.\ such that $|h_s(x)|=1$ a.e.
 If we extend $h_s$ to the whole $\R$ as $1$-periodic functions, the result implies that the system
$\{ h_s \tau_t \1_{[0,1]} \}$, $(t,s) \in \Z \times \Z$, forms an unconditional Schauder frame in the space $L^p(\R)$. Hence, contrary to the previous section, \thmref{thm:gabframesp<2} becomes false if we replace the exponential functions $e_s$ with arbitrary unimodular functions $h_s$, i.e.\ the proof must involve more specific properties of the exponentials.

We note that the proof below works also in the more general case where $\TTT \sbt \R$ is a set of bounded density, i.e.\ a finite union of uniformly discrete sets.

\subsection{} 
First of all, observe that when we deal with unconditional Schauder frames, we can always add extra elements to the system with zeros as coefficient functionals. Hence, it would suffice to prove that a Gabor system $\G(g; \TTT\times \SSS)$ can never form an unconditional Schauder frame in the space $L^p(\R)$, $1< p < 2$, where $g\in L^p(\R)$,
$\TTT\subset\R$ is a  uniformly discrete set (or, more generally, a finite union of uniformly discrete sets), and $\SSS\subset\R$ is an arbitrary countable set.

Suppose to the contrary that every function $f\in L^p(\R)$ can be represented as an unconditionally convergent series
\begin{equation}\label{eq:4.1}
f = \sum_{t\in \TTT} \sum_{s\in \SSS} \gtss(f) e_s \ts_t g.
\end{equation}
The proof will consist of two steps: at first we obtain an estimate on the sequence of coefficients $\{\gtss(f)\}$, and then use this estimate to arrive at a contradiction.

\subsection{The estimate on the coefficients}
Without any loss of generality, we may assume that $\|g\|_p=1$. An application of \lemref{lem:unc_constant} together with cotype 2 inequality (\lemref{lem:cotype}) then yields the estimate
\begin{equation}
\Big(\sum_{t\in \TTT} \sum_{s\in \SSS} |g_{ts}^*(f)|^2\Big)^{1/2}\le K C_p  \|f\|_p.
\end{equation}
However, this estimate is not enough for us if the set $\SSS$ is not uniformly discrete, and therefore, similarly to the idea from \cite{LT25a}, we will prove a stronger estimate on the coefficients $\{g_{ts}^*(f)\}$.

To this end, we choose and fix a small number $\del > 0$ such that
\begin{equation}
\label{eq:gsup}
\sup_{0 < s < \del}
\int_{\R} |g(x)|^p \cdot |e^{2 \pi i s x} - 1|^p   dx < 1/2^p.
\end{equation}
After that, we divide the set $\SSS$ into blocks: for $l\in\Z$ we put
\begin{equation}
\label{eq:sl.block}
\SSS_l = \SSS\cap [l\delta, (l+1)\delta).
\end{equation}
An application of \lemref{lem:unc_constant} yields that for arbitrary signs $\theta_{tl} = \pm 1$ the following estimate holds for every $f\in L^p(\R)$,
\begin{equation}
\Big\| \sum_{t\in \TTT} \sum_{l\in\Z} \theta_{tl} \sum_{s\in \SSS_l} |g_{ts}^*(f)|e^{-2\pi i s t} e_s \ts_t g \Big\|_p\le K\|f\|_p.
\end{equation}
Putting now a Rademacher sequence $\eps_{tl}$ instead of $\theta_{tl}$, taking expectation and applying cotype 2 inequality (\lemref{lem:cotype}), we get that
\begin{equation}\label{eq:4.6}
\Big(\sum_{t\in \TTT} \sum_{l\in \Z} \Big\| \sum_{s\in \SSS_l} |g_{ts}^*(f)|e^{-2\pi i s t}e_s \ts_t g \Big\|_p^2\Big)^{1/2}\le KC_p\|f\|_p.
\end{equation}

Now, we claim that if $s_1, s_2\in \SSS_l$ then
\begin{equation}\label{eq:4.7}
\|e^{-2\pi i s_1 t} e_{s_1} \ts_t g - e^{-2\pi i s_2 t} e_{s_2} \ts_t g\|_p < 1/2.
\end{equation}
Indeed, we have
\begin{align}
&\|e^{-2\pi i s_1 t} e_{s_1} \ts_t g - e^{-2\pi i s_2 t} e_{s_2} \ts_t g\|^p_p=\int_\R |g(x-t)|^p\cdot |e^{2\pi i s_1 (x-t)} - e^{2\pi i s_2(x-t)}|^p dx\\
& \quad =\int_\R |g(x)|^p\cdot |e^{2\pi i s_1 x} - e^{2\pi i s_2 x}|^p dx = \int_\R |g(x)|^p\cdot |e^{2\pi i (s_1-s_2)x}-1|^p dx  < 1/2^p,
\end{align} 
where the last inequality follows from \eqref{eq:gsup}, \eqref{eq:sl.block}. Hence \eqref{eq:4.7} is proved.

Next, for each $l$ such that $\SSS_l$ is nonempty, we choose and fix some element $s_l\in \SSS_l$. Then for each $t \in T$, by the triangle inequality and the estimate \eqref{eq:4.7},
\begin{align}
& \Big\| \sum_{s\in \SSS_l} |g_{ts}^*(f)|e^{-2\pi i s t} e_s \ts_t g \Big\|_p  \ge \Big\| \Big( \sum_{s\in \SSS_l} |g_{ts}^*(f)| \Big) e^{-2\pi i s_l t} e_{s_l} \ts_t g \Big\|_p \\
& \qquad \qquad - \Big\| \sum_{s\in \SSS_l} |g_{ts}^*(f)|(e^{-2\pi i s_l t}e_{s_l} \ts_t g - e^{-2\pi i s t} e_s \ts_t g )\Big\|_p\\
& \qquad \ge \sum_{s\in \SSS_l} |g_{ts}^*(f)| - \sum_{s\in \SSS_l} |g_{ts}^*(f)| \cdot \|e^{-2\pi i s_l t}e_{s_l} \ts_t g - e^{-2\pi i s t} e_s \ts_t g\|_p \\
& \qquad \ge \frac{1}{2} \sum_{s\in \SSS_l} |g_{ts}^*(f)|.
\end{align}
It remains to plug this estimate into \eqref{eq:4.6} and arrive at the following key estimate on the coefficients $\{g_{ts}^*(f)\}$, which holds for an arbitrary $f\in L^p(\R)$,
\begin{equation}
\label{eq:est_on_coeff}
\Big(\sum_{t\in \TTT} \sum_{l\in \Z} \Big( \sum_{s\in \SSS_l} |g_{ts}^*(f)| \Big)^2 \Big)^{1/2}\le 2KC_p \|f\|_p.
\end{equation}

\subsection{The contradiction}
For each $t \in T$, we let
\begin{equation}
\label{eq:utf}
(u_{t} f)(x) = \sum_{s\in \SSS} g_{ts}^*(f) e_s(x), \quad x \in \R.
\end{equation}
According to \lemref{lem:bess}, this series is unconditionally convergent locally in $L^2$, and moreover, combining the estimates \eqref{eq:as.es} and \eqref{eq:est_on_coeff} we get that for any interval $I\subset\R$ of unit length, we have
\begin{equation}\label{eq:4.19}
\sum_{t\in \TTT} \|u_{t}f\|_{L^2(I)}^2 \lesssim \|f\|^2_p
\end{equation}
where the implied constant depends neither on $f$ nor on the interval $I$.

It follows from the unconditional convergence of the series expansion \eqref{eq:4.1} that for each $t \in \TTT$, the series
$ (v_t f)(x) = \sum_{s\in \SSS} \gtss(f) e_s (x) g(x-t)  $ converges unconditionally in the space $L^p(\R)$. By passing to subsequences and using pointwise a.e.\ convergence, we see that $(v_t f)(x) = g(x-t) (u_t f)(x) $ a.e. Hence, any $f\in L^p(\R)$ admits an expansion
\begin{equation}
\label{eq:4.20}
f(x) = \sum_{t\in \TTT} g(x-t) (u_{t}f)(x)
\end{equation}
which converges unconditionally in $L^p(\R)$.

Next, consider the function
\begin{equation}
G(x)=  \Big(\sum_{t\in \TTT}|g(x-t)|^p\Big)^{1/p}, \quad x \in \R.
\end{equation}
For any interval $I\subset\R$ of unit length, we have
\begin{equation}
\label{eq:G.LP}
\int_I G(x)^p dx = \sum_{t\in \TTT} \int_I|g(x-t)|^p dx = \int_{\R}|g(x)|^p  \sum_{t\in \TTT} \1_{I-t}(x)  dx.
\end{equation}
If $\TTT$ is a uniformly discrete set (or, more generally, a finite union of uniformly discrete sets), then the function  $\sum_{t\in \TTT} \1_{I-t}$ is uniformly bounded on $\R$. Hence, the integral on the right hand side of \eqref{eq:G.LP} converges, which means that $G\in L^p(I)$. 

We now use the functions $u_t f$, $t \in \TTT$, and $G$ in order to establish a pointwise estimate on the function $f$. It follows from the expansion \eqref{eq:4.20} (by passing to a subsequence and using pointwise a.e.\ convergence) that
\begin{align}
|f(x)| &\le \sum_{t\in \TTT} |(u_{t}f)(x)| \cdot |g(x-t)|  \le \Big(\sum_{t\in \TTT} |(u_{t} f)(x)|^{p'}\Big)^{1/p'}\Big( \sum_{t\in \TTT} |g(x-t)|^p \Big)^{1/p}
\label{eq:fug.1} \\
& \le \Big( \sum_{t\in \TTT} |(u_{t}f)(x)|^2 \Big)^{1/2}  G(x) \quad \text{a.e.,}
\label{eq:fug.2}
\end{align}
where $p'= p / (p-1) $ is the exponent conjugate to $p$. In the last inequality we used the assumption that $1<p<2$, and hence $p' > 2$.

Finally, consider an arbitrary unit interval $I$. Since $G\in L^p(I)$, there exists a positive constant $M$ and a measurable set $E \subset I$, $m(E)>0$, such that $G(x)\le M$ a.e.\ on $E$. Hence, if $f\in L^p(\R)$ and $f(x) =0$ a.e.\ on  $\R\setminus E$, then \eqref{eq:fug.1}--\eqref{eq:fug.2} imply that
\begin{equation}
|f(x)|^2\le M^2  \sum_{t\in \TTT} |(u_{t} f)(x)|^2  \quad \text{a.e.}
\end{equation}
By integrating this inequality over the interval $I$ and using \eqref{eq:4.19}, we conclude that the estimate $\|f\|_2\lesssim \|f\|_p$ holds for any $f\in L^p(\R)$ supported on the set $E$. Since $1 < p < 2$, this gives us the desired contradiction and completes the proof of \thmref{thm:gabframesp<2}.
\qed

% =======================================

\section{The case $p > 2$: Unconditional Gabor basic sequences}
\label{sec:basesp>2}

Now we turn to the proof of \thmref{thm:nobases_p>2}. Our first step is to formulate a statement about Gabor unconditional basic sequences similar to \thmref{thm:uncseq_p<2}, which would imply \thmref{thm:nobases_p>2}.

\subsection{}
Assume that we are given a Gabor system $\G(g; \Lambda)$ in the space $L^p(\R)$. We choose and fix a sufficiently small number $\delta = \delta(g,p) > 0$ satisfying 
\begin{equation}
\label{eq:5.d.1}
\sup_{0 < t < \delta}
\|\tau_t g - g\|_p \le \tfrac{1}{2} \|g\|_p.
\end{equation}
Next, we partition the set $\Lambda \sbt \R \times \R$
into vertical strips of width $\delta$, i.e.\ we define
\begin{equation}
\label{eq:5.1}
\Lambda_k = \{(t,s)\in \Lambda:  k\delta \le t < (k+1)\delta \}, \quad k \in \Z.
\end{equation}
We will  prove the following theorem.

\begin{thm}
\label{thm:uncseq_p>2}
Let $2 \le p < \infty$, and suppose that $\G(g; \Lambda)$ forms an unconditional basic sequence in $L^p(\R)$, where $g\in L^p(\R)$ and  $\Lambda \sbt \R \times \R$ is any countable set. Assume that $X = \cspan \, \G(g; \Lambda)$ is a complemented subspace of $L^p(\R)$. If  $\Lambda_k$ are defined by \eqref{eq:5.d.1}, \eqref{eq:5.1}, then the system $\G(g; \Lambda)$ is equivalent to the standard unit vector basis of $(\bigoplus_{k\in\Z} \ell^2_{\Lambda_k})_p$.
\end{thm}

The equivalence stated in this theorem should be understood similarly to \thmref{thm:uncseq_p<2}: the elements of the standard unit vector basis of the space $(\bigoplus_{k\in\Z} \ell^2_{\Lambda_k})_p$ are naturally indexed by the points of $\Lambda$; then the conclusion of the theorem means that the  linear operator which maps $e_s \ts_t g$, $(t,s)\in \Lambda$, to the corresponding element of the unit vector basis of the space $(\bigoplus_{k\in\Z} \ell^2_{\Lambda_k})_p$ is an isomorphism between the subspace $X = \cspan \,  \G(g;\Lambda)$ and $(\bigoplus_{k\in\Z} \ell^2_{\Lambda_k})_p$. In particular, for $p > 2$ we obtain that $X \neq L^p(\R)$ by \lemref{lem:noniso}, and as a consequence, \thmref{thm:nobases_p>2} follows.

There is an important difference, however, between Theorems \ref{thm:uncseq_p<2} and \ref{thm:uncseq_p>2}: note that in \thmref{thm:uncseq_p>2} there is an extra requirement that $X = \cspan \, \G(g; \Lambda)$ be a complemented subspace of $L^p(\R)$. This requirement, in fact, is crucial: we will present an example showing that without this assumption the result fails
(see \thmref{thm:n.c.s} below).

On the other hand, note that in \thmref{thm:uncseq_p>2} (contrary to 
\thmref{thm:uncseq_p<2}) we do not impose any
a priori assumption on the countable set
$\Lam \sbt \R \times \R$.

One of the difficulties in the proof of \thmref{thm:uncseq_p>2}
compared to \thmref{thm:uncseq_p<2},  is that \lemref{lem:JO.EXT} is not true for $p > 2$
(see \cite[Example 2.16]{OSSZ11}). We overcome this obstacle by using duality, similar to the proof of \cite[Theorem 2.1]{FOSZ14}. However, in our case certain additional ideas are required.

The proof of \thmref{thm:uncseq_p>2} remains valid if we replace the complex exponentials $e_s$ with arbitrary unimodular functions $h_s$, or, more generally, with arbitrary functions $h_s$ such that $A \le |h_s(x)| \le B$ a.e. for some positive constants $A,B$.

\subsection{}
We now turn to prove \thmref{thm:uncseq_p>2}. Suppose that $\G(g; \Lambda)$ is a $K$-unconditional basic sequence in $L^p(\R)$, $p \ge 2$. Without loss of generality, we can assume that $\|g\|_p = 1$. Let the sets $\Lambda_k$ be defined by \eqref{eq:5.d.1}, \eqref{eq:5.1}. 

In order to prove our theorem, we need to show
that for any sequence of scalars $\{a_{ts}\}$, $(t,s)\in \Lambda$, 
with only finitely many nonzero terms, we have
\begin{equation}
\label{eq:toprove_bas}
\Big\| \sum_{k\in\Z} \sum_{(t,s)\in\Lambda_k} a_{ts} e_s \ts_t g \Big\|_p \asymp \Big(\sum_{k\in\Z}\Big( \sum_{(t,s)\in\Lambda_k}|a_{ts}|^2 \Big)^{p/2} \Big)^{1/p}.
\end{equation}
Of course, the implied constants should not depend on the coefficients $\{a_{ts}\}$.

The proof of the estimate \eqref{eq:toprove_bas} will be done in several steps. 

\subsubsection{Step 1: Reformulation via duality}
Let $X= \cspan\,  \G(g;\Lambda)$, so that $X$ is a closed subspace of
$L^p(\R)$, and the system
$\G(g; \Lambda)$ forms a $K$-unconditional basis of $X$.
Let  $\{\pphi_{ts}^*\} \subset X^*$ be the 
system of functionals biorthogonal to $\{e_s\tau_t g\}$, $(t,s)\in\Lambda$. 
We claim that  \eqref{eq:toprove_bas} would follow if we prove that for any sequence of scalars 
$\{b_{ts}\}$, $(t,s)\in \Lambda$, with only finitely many nonzero terms, we have
\begin{equation}
\label{eq:toprove_bas_dual}
\Big\| \sum_{k\in\Z} \sum_{(t,s)\in\Lambda_k} b_{ts} \pphi_{ts}^* \Big\|_{X^*} \asymp \Big(\sum_{k\in\Z}\Big( \sum_{(t,s)\in\Lambda_k}|b_{ts}|^2 \Big)^{p'/2} \Big)^{1/p'},
\end{equation}
where $p' = p/(p-1)$. 

This claim incorporates the fact that the space $\ell^{p'}(\ell^2)$ is 
the dual of $\ell^p(\ell^2)$; however we will not use this
fact explicitly and rather provide a direct proof of the claim.
To prove the claim, we need the following observation.

\begin{lem}
\label{lem:dual.k}
The system $\{\pphi_{ts}^*\}$,  $(t,s)\in \Lambda$,  forms a
$K$-unconditional basis of $X^*$.
\end{lem}

\begin{proof}
Since the system $\{e_s\tau_t g\}$
  forms a $K$-unconditional basis of $X$, 
the biorthogonal system $\{\pphi_{ts}^*\}$ forms a
$K$-unconditional basic sequence in $X^*$
(this can be proved by a standard duality argument). Moreover, 
 $X$ is a closed subspace of $L^p(\R)$, $p \ge 2$, and the space $L^p(\R)$ is reflexive, hence $X$ is a reflexive Banach space (see \cite[Theorem II.A.14]{Woj91}).
Finally, since $X$ is reflexive, the biorthogonal system
$\{\pphi_{ts}^*\}$ is complete in the space $X^*$ 
(see \cite[Proposition 3.2.6]{AK16}),
so the conclusion of the lemma follows.
\end{proof}

We now turn to prove the claim, that is, we show
that \eqref{eq:toprove_bas_dual} implies \eqref{eq:toprove_bas}.
Suppose that \eqref{eq:toprove_bas_dual} holds,
and let $\{a_{ts}\}$, $(t,s)\in \Lambda$, be a sequence of
scalars with only finitely many nonzero terms. 
In order to establish \eqref{eq:toprove_bas}, we set
\begin{equation}
h =   \sum_{k \in \Z} \sum_{(t,s)\in\Lambda_k}   a_{ts} e_s \ts_t g,
\end{equation}
and we need to estimate the norm $\|h\|_p$.

To estimate $\|h\|_p$ from above, we use the fact that
\begin{equation}
\label{eq:hks.n.d}
\|h \|_{p} = \sup \{ | \pphi(h) | : \pphi \in X^*, \|\pphi\|_{X^*} \le 1 \}.
\end{equation}
Since $ \{\pphi^*_{ts}\}$, $(t,s)\in\Lambda$, is a complete system
in the space $X^*$ (due to \lemref{lem:dual.k}),
it suffices to take the supremum in \eqref{eq:hks.n.d}
over finite linear combinations of the elements 
$ \{\pphi^*_{ts}\}$.  So, suppose that $\pphi \in X^*$,
 $\| \pphi \|_{X^*} \le 1$, is given by
$\pphi = \sum_{(t,s)\in\Lambda}  b_{ts} \pphi_{ts}^*$
where only finitely many coefficients $b_{ts}$ are nonzero.
By an application of the Cauchy-Schwarz inequality and H\"{o}lder's inequality, we get
\begin{align}
 |\pphi(h)| &=
  \Big| \sum_{k \in \Z} \sum_{(t,s)\in\Lambda_k}  a_{ts} b_{ts} \Big|
 \le \sum_{k \in \Z} \Big\{ \Big( \sum_{(t,s)\in\Lambda_k}  |a_{ts}|^2 \Big)^{1/2}
 \Big( \sum_{(t,s)\in\Lambda_k}  |b_{ts}|^2 \Big)^{1/2} \Big\} \label{eq:cs.h.1} \\
& \le \Big( \sum_{k \in \Z} \Big( \sum_{(t,s)\in\Lambda_k}  |a_{ts}|^2 \Big)^{p/2} \Big)^{1/p}
 \Big( \sum_{k \in \Z} \Big( \sum_{(t,s)\in\Lambda_k}  |b_{ts}|^2 \Big)^{p'/2} \Big)^{1/p'} \label{eq:cs.h.2} \\
& \lesssim \Big( \sum_{k \in \Z} \Big( \sum_{(t,s)\in\Lambda_k}  |a_{ts}|^2 \Big)^{p/2} \Big)^{1/p}, \label{eq:cs.h.3}
\end{align}
where, in order to pass from \eqref{eq:cs.h.2}  to \eqref{eq:cs.h.3},
we used the estimate \eqref{eq:toprove_bas_dual}  and the fact that  $\|\pphi\|_{X^*} \le 1$.
Hence, combining  \eqref{eq:hks.n.d}, \eqref{eq:cs.h.1}, \eqref{eq:cs.h.2}, \eqref{eq:cs.h.3}
we obtain the upper norm estimate in \eqref{eq:toprove_bas}.

Next, we estimate $\|h\|_p$ from below. 
We choose scalars $\{b_{ts}\}$, $(t,s)\in \Lambda$, such 
that\footnote{The choice of the scalars $\{b_{ts}\}$ 
uses the fact that $ (\ell^p)^* = \ell^{p'} $
and $ (\ell^2)^*   = \ell^{2} $, and is done in
the following way. First,  put 
$c_k = ( \sum_{(t,s) \in \Lam_k} |a_{ts}|^2 )^{1/2}$, and
 choose scalars $\{d_k\}$ satisfying $\sum_k |d_k|^{p'} = 1$ and 
$\sum_k c_k d_k  = ( \sum_k |c_k|^p )^{1/p}$.
Then,   for each $k$, 
 choose scalars $\{b_{ts}\}$, $(t,s) \in \Lam_k$, 
such that $  \sum_{(t,s) \in \Lam_k} |b_{ts}|^2  = |d_k|^2$
and $ \sum_{(t,s) \in \Lam_k} a_{ts} b_{ts} = c_k d_k$.
This choice satisfies both \eqref{eq:lpl2:1} and \eqref{eq:lpl2:2}.} 
\begin{equation}
\label{eq:lpl2:1}
 \sum_{k \in \Z} \Big( \sum_{(t,s)\in\Lambda_k}  |b_{ts}|^2 \Big)^{p'/2} = 1,
\end{equation}
and
\begin{equation}
\label{eq:lpl2:2}
 \sum_{k \in \Z} \sum_{(t,s)\in\Lambda_k}  a_{ts} b_{ts} =
 \Big( \sum_{k \in \Z} \Big( \sum_{(t,s)\in\Lambda_k}  |a_{ts}|^2 \Big)^{p/2} \Big)^{1/p}
\end{equation}
(this choice, in particular, makes the inequalities in
\eqref{eq:cs.h.1}, \eqref{eq:cs.h.2} become equalities),
and we observe that only finitely many of the scalars $b_{ts}$  are nonzero.
Hence, we may define an element $\pphi \in X^*$ by
 $\pphi = \sum_{(t,s)\in\Lambda}  b_{ts} \pphi_{ts}^*$, and we have
\begin{equation}
\label{eq:lpl2:3}
 \Big| \sum_{k \in \Z} \sum_{(t,s)\in\Lambda_k}  a_{ts} b_{ts} \Big|
 = |\pphi(h)|  \le \|h\|_p \|\pphi \|_{X^*}.
\end{equation}
Thus, \eqref{eq:toprove_bas_dual}, 
\eqref{eq:lpl2:1}, \eqref{eq:lpl2:2}, \eqref{eq:lpl2:3}
imply the lower estimate in \eqref{eq:toprove_bas}.

To sum up, we proved that \eqref{eq:toprove_bas_dual} implies \eqref{eq:toprove_bas}.
It thus remains to establish \eqref{eq:toprove_bas_dual}.

\subsubsection{Step 2: An estimate for one block in $L^p$}
Fix $k\in\Z$ such that the set $\Lambda_k$ is nonempty.
Assume that $\{  a_{ts} \}$, $(t,s)\in\Lambda_k$, 
are scalars such that only finitely many of them are nonzero,
and consider the function
\begin{equation}\label{eq:hk}
h_{k} =  \sum_{(t,s)\in\Lambda_k}  a_{ts} e_s \ts_t g.
\end{equation}
Our goal is to prove that the estimate
 \eqref{eq:toprove_bas} holds for the function
 $h_k$.
 
We start by estimating $\|h_k\|_p$ from above.
Recall that $\G(g; \Lambda)$ forms a $K$-unconditional basic sequence
 in the space $L^p(\R)$, and $\|g\|_p = 1$. Hence, 
if  we let $\{ \eps_{ts} \}$,  $(t,s)\in\Lambda_k$,  be a Rademacher sequence, then
by applying type 2 inequality (\lemref{lem:type}) we get
\begin{equation}
\|h_{k}\|_p \le K \cdot \mathbb{E} \Big\|\sum_{(t,s)\in\Lambda_k} \eps_{ts} a_{ts} e_s \ts_t g \Big\|_p\le K C_p \Big( \sum_{(t,s)\in\Lambda_k}|a_{ts}|^2 \Big)^{1/2}.
\end{equation}

Next, we turn to estimate $\|h_k\|_p$ from below. First, we observe that
\begin{align}
\|h_{k}\|_p&\ge K^{-1} \Big(\mathbb{E} \Big\|\sum_{(t,s)\in\Lambda_k} \eps_{ts}  a_{ts} e_s \ts_t g \Big\|^p_p\Big)^{1/p} \label{eq:5.10.a} \\
&\ge K^{-1} A_p \Big\| \Big( \sum_{(t,s)\in\Lambda_k}|a_{ts}|^2 |\ts_t g|^2 \Big)^{1/2} \Big\|_p,
\label{eq:5.11}
\end{align}
where the last inequality follows from \lemref{lem:squarefunc}.

Now, we introduce some notation. Let
$\TTT_k$ denote the projection of $\Lambda_k$ on the $x$-axis, 
namely, $\TTT_k = \{t :  (t,s)\in\Lambda_k\}$; we observe that 
$\TTT_k  \sbt [k\delta, (k+1)\delta)$
which follows from \eqref{eq:5.1}.
 Next, for each point $t \in \TTT_k $  we use
 $\SSS_t$ to denote the ``fiber'' of $\Lam_k$ 
above the point $t$, that is,
$\SSS_t=\{s\in\R: (t,s)\in\Lambda_k\}$.
Finally, we also denote 
\begin{equation}
\label{eq:ct.def}
c_t = \Big( \sum_{s\in \SSS_t} |a_{ts}|^2 \Big)^{1/2}, \quad t \in \TTT_k.
\end{equation}
The quantity \eqref{eq:5.11} can be written using this notation as
\begin{equation}
\label{eq:5.10.b}
K^{-1} A_p \Big\| \Big( 
\sum_{t\in \TTT_k}|c_{t}|^2  |\ts_t g|^2 \Big)^{1/2} \Big\|_p
 = K^{-1} A_p  \|(c_t \ts_t g)_t\|_{L^p(\ell^2)}.
\end{equation}

Next, we choose and fix an arbitrary point
 $t_k \in \TTT_k$. By the triangle inequality in the space
$L^p(\ell^2)$, and using again that $\|g\|_p = 1$, we get
\begin{align}
&\|(c_t \ts_t g)_t\|_{L^p(\ell^2)} \ge 	\|(c_t \ts_{t_k} g)_t\|_{L^p(\ell^2)} - \|(c_t (\ts_t g - \ts_{t_k} g))_t\|_{L^p(\ell^2)} \label{eq:lpl2.1} \\
& \qquad \ge \Big( \sum_{t\in \TTT_k} |c_t|^2\Big)^{1/2} - \Big( \sum_{t\in \TTT_k} \|c_t  (\ts_t g - \ts_{t_k} g)\|_p^2 \Big)^{1/2}\ge \frac{1}{2}\Big( \sum_{t\in \TTT_k} |c_t|^2\Big)^{1/2}, \label{eq:lpl2.2}
\end{align}
where we used \lemref{lem:mink} in order to pass
from \eqref{eq:lpl2.1} to \eqref{eq:lpl2.2}; while in the last inequality we used the property \eqref{eq:5.d.1} satisfied by our choice of $\delta$, which implies that for each $t \in \TTT_k$ we have $\|\ts_t g - \ts_{t_k} g \|_p  \le 1/2$.

Clearly, it follows from \eqref{eq:ct.def} that
\begin{equation}
\label{eq:5.10.c}
\sum_{t\in \TTT_k} |c_t|^2 = \sum_{(t,s)\in\Lambda_k}|a_{ts}|^2.
\end{equation}
Hence, if we combine \eqref{eq:5.10.a}, \eqref{eq:5.11}, \eqref{eq:5.10.b}, 
\eqref{eq:lpl2.1}, \eqref{eq:lpl2.2}, \eqref{eq:5.10.c}, then
we obtain a lower estimate for the norm $\|h_k\|_p$.

Summing up, in this step we have shown that for any $k\in\Z$ and for any function $h_k \in L^p(\R)$ given by \eqref{eq:hk}, we have the estimate
\begin{equation}
\label{eq:5.17}
(2K)^{-1} A_p \Big( \sum_{(t,s)\in\Lambda_k}|a_{ts}|^2 \Big)^{1/2} \le \|h_{k}\|_p\le K C_p \Big( \sum_{(t,s)\in\Lambda_k}|a_{ts}|^2 \Big)^{1/2}.
\end{equation}

\subsubsection{Step 3: An estimate for one block in $L^{p'}$}
Recall that $\{\pphi_{ts}^*\} \subset X^*$ is the
system of functionals biortho\-gonal to the system $\{e_s\tau_t g\}$, $(t,s)\in\Lambda$. 
The system $\{\pphi_{ts}^*\}$,  $(t,s)\in \Lambda$,  forms a
$K$-unconditional basis of $X^*$  according to \lemref{lem:dual.k}.
Also, we note that
\begin{equation}
\label{eq:pphits.k}
1 \le  \| \pphi_{ts}^* \|_{X^*} \le K, \quad (t,s)\in \Lambda.
\end{equation}
Indeed, since $\|g\|_p = 1$,  we have on one hand,
$1 = |\pphi_{ts}^* (e_s \ts_t g)| \le \|\pphi_{ts}^*\|_{X^*}$;
while on the other hand,
any element $h \in X$ admits a unique (unconditionally convergent) 
expansion as $h = \sum_{(t,s)\in\Lambda}  a_{ts} e_s \ts_t g$,
and hence
$|\pphi_{ts}^*(h)| = | a_{ts}| = \|  a_{ts} e_s \ts_t g \|_p \le K \|h\|_p$.

By assumption, $X= \cspan \,  \G(g;\Lambda)$
is a complemented subspace of $L^p(\R)$. This means that
there is a bounded projection  operator $P$ from
$L^p(\R) $  onto $ X$. Let $L = \|P\|$.
The dual space $(L^p(\R))^*$ 
can be identified with $L^{p'}(\R)$, $p' = p/(p-1)$,
in the usual way. Hence, the 
dual operator $P^* : X^* \to L^{p'}(\R)$
is an isomorphic embedding satisfying 
\begin{equation}
\label{eq:dual_proj}
	\|x^*\|_{X^*} \le\|P^*(x^*)\|_{p'} \le L \|x^*\|_{X^*}, \quad x^* \in X^*.
\end{equation}

We now define a system of functions $\{ g_{ts}^* \} \sbt L^{p'}(\R)$ given by
\begin{equation}
\label{eq:gts.s.d}
g_{ts}^*=P^*(\pphi_{ts}^*), \quad (t,s)\in\Lambda. 
\end{equation}
It follows from \eqref{eq:pphits.k}, \eqref{eq:dual_proj}, \eqref{eq:gts.s.d} that
\begin{equation}
\label{eq:gts.s.norm}
1 \le \|g_{ts}^*\|_{p'} \le L K.
\end{equation}

We also need the following observation.

\begin{lem}
\label{lem:dual.gts.s}
The system $\{ g_{ts}^* \}$, $(t,s) \in\Lambda$, forms an $(LK)$-unconditional 
basic sequence in the space $L^{p'}(\R)$.
\end{lem}

\begin{proof}
 Indeed, for any sequence of scalars $\{b_{ts}\}$, $(t,s)\in\Lambda$, 
 with only finitely many nonzero terms, and for arbitrary scalars
  $\theta_{ts}$ with $|\theta_{ts}|\le 1$, we have
\begin{align}
& \Big\| \sum_{(t,s)\in\Lambda} \theta_{ts} b_{ts} g^*_{ts} \Big\|_{p'} = \Big\| P^* \Big( \sum_{(t,s)\in\Lambda} \theta_{ts} b_{ts} \pphi^*_{ts} \Big)\Big\|_{p'}\le L \Big\|\sum_{(t,s)\in\Lambda} \theta_{ts} b_{ts} \pphi^*_{ts} \Big\|_{X^*}\\
&\quad \le L K \Big\| \sum_{(t,s)\in\Lambda}  b_{ts} \pphi^*_{ts} \Big\|_{X^*} \le L K \Big\|P^*\Big(\sum_{(t,s)\in\Lambda}  b_{ts} \pphi^*_{ts} \Big)\Big\|_{p'} = L K \Big\| \sum_{(t,s)\in\Lambda}  b_{ts} g^*_{ts} \Big\|_{p'},
\end{align}
where we used \eqref{eq:dual_proj} in the first and third inequalities.
This establishes our claim.
\end{proof}

Now, fix $k \in \Z$ such that the set
$\Lam_k$ is nonempty. Assume that  $\{b_{ts}\}$, $(t,s)\in\Lambda_k$, 
are scalars such that only finitely many of them are
nonzero, and consider the function
\begin{equation}
\label{eq:hks.def}
h_k^* = \sum_{(t,s)\in\Lambda_k} b_{ts}g^*_{ts}.
\end{equation}
Our next goal is to establish an estimate similar to
\eqref{eq:5.17} for the norm $\| h_k^* \|_{p'}$.

We first estimate $\|h_k^*\|_{p'}$ from below. 
Let $\{\eps_{ts}\}$ be a Rademacher sequence. We use
the fact that $1 < p' \le 2$, so that
cotype 2 inequality (\lemref{lem:cotype}) holds
in the space $L^{p'}(\R)$. Together with
\eqref{eq:gts.s.norm} and \lemref{lem:dual.gts.s}, this yields
\begin{equation}
\label{eq:belowest}
\|h_k^*\|_{p'} \ge (L K)^{-1} \,
\mathbb{E} \Big\|\sum_{(t,s)\in\Lambda_k} 
\eps_{ts}b_{ts}  g_{ts}^* \Big\|_{p'}\ge (L K C_{p'})^{-1}
\Big( \sum_{(t,s)\in\Lambda_k}|b_{ts}|^2 \Big)^{1/2},
\end{equation}
where $C_{p'}$ is the constant in cotype 2 inequality for the space $L^{p'}(\R)$, see \lemref{lem:cotype}.

Next, we estimate $\|h_k^*\|_{p'}$ from above. 
We observe that 
\begin{equation}
\label{eq:hk.psik.d}
h_k^* = P^*(	\psi_{k}^*), \quad \psi_{k}^* = \sum_{(t,s)\in\Lambda_k} b_{ts} \pphi^*_{ts},
\end{equation}
and hence, by \eqref{eq:dual_proj}, we have 
\begin{equation}
\label{eq:dualnorm}
\|h_k^*\|_{p'}\le L \|\psi_k^*\|_{X^*} = L  \sup \{ |\psi_k^*(h)| : 
h \in X, \|h\|_{p} \le 1\}.
\end{equation}
Since $ \{e_s\tau_t g\}$, $(t,s)\in\Lambda$, is a complete system
in the space $X$, it suffices to take the supremum in \eqref{eq:dualnorm} over finite linear combinations of the functions $e_s\tau_t g$. So, suppose that $h \in X$, $\| h \|_p \le 1$, is a function of the form
\begin{equation}
\label{eq:h.hj.d}
h = \sum_{j \in \Z} h_j, \quad h_j =  \sum_{(t,s)\in\Lambda_j } a_{ts} e_s\tau_t g,
\end{equation}
where only finitely many coefficients $a_{ts}$ are nonzero.
Using $K$-unconditionality of the system 
$\{e_s\tau_t g\}$, $(t,s)\in\Lambda$,
 and then the estimate \eqref{eq:5.17}, we conclude that
\begin{equation}
1 \ge \|h\|_X \ge K^{-1} \|h_k\|_{X}
\ge (2K^2)^{-1} A_p  \Big( \sum_{(t,s)\in\Lambda_k} |a_{ts}|^2 \Big)^{1/2},
\end{equation}
and hence
\begin{equation}
\label{eq:hk.test}
\Big(\sum_{(t,s)\in\Lambda_k} |a_{ts}|^2 \Big)^{1/2} \le  2K^2 A_p^{-1}.
\end{equation}
It remains to apply the Cauchy--Schwartz inequality, i.e.\ using
\eqref{eq:hk.psik.d}, \eqref{eq:h.hj.d} we get
\begin{align}
|\psi_k^*(h)|  &= \Big| \sum_{(t,s)\in\Lambda_k} a_{ts} b_{ts} \Big| 
\le \Big( \sum_{(t,s)\in\Lambda_k} |a_{ts}|^2 \Big)^{1/2} 
\Big( \sum_{(t,s)\in\Lambda_k} |b_{ts}|^2 \Big)^{1/2} \label{eq:psk.e.1} \\
& \le  2K^2 A_p^{-1} \Big( \sum_{(t,s)\in\Lambda_k} |b_{ts}|^2 \Big)^{1/2},  \label{eq:psk.e.2}
\end{align}
where in the last inequality we used \eqref{eq:hk.test}. Thus, 
by combining \eqref{eq:dualnorm}, \eqref{eq:psk.e.1}, \eqref{eq:psk.e.2},
we obtain an estimate from above for the norm $\|h_k^*\|_{p'}$.

To sum up, we have shown that
for any $k\in\Z$ and any function $h^*_k \in L^{p'}(\R)$ given by 
\eqref{eq:hks.def}, we have the estimate
\begin{equation}
\label{eq:keyest_dual}
(L K C_{p'})^{-1}  \Big( \sum_{(t,s)\in\Lambda_k}|b_{ts}|^2 \Big)^{1/2}
\le \|h_k^*\|_{p'}
\le  2 L K^2  A_p^{-1} \Big( \sum_{(t,s)\in\Lambda_k} |b_{ts}|^2 \Big)^{1/2}.
\end{equation}

\emph{Remark}. In fact, we will only need the estimate from above
in \eqref{eq:keyest_dual}. We proved also the estimate from below
in order to provide the reader with a more complete picture.

\subsubsection{Step 4: A concentration inequality}

The idea now is to apply \lemref{lem:JO.EXT} in the space
$L^{p'}(\R)$, using the fact that $1 < p' \le 2$. However,
a difficulty arises:  since each ``block'' $h_k^*$ is a linear
combination of several elements
of the system $\{ g^*_{ts}\}$, $(t,s) \in \Lambda_k$, 
we cannot establish a good lower estimate for
$\|h^*_k|_{E_k}\|_{p'}$ where $E_k \sbt \R$ are
measurable sets satisfying the condition
$\sum_k \1_{E_k} \le M$ a.e.
Indeed,  while we do know that the
system $\{ g^*_{ts}\}$, $(t,s) \in \Lambda_k$, 
forms an unconditional basic sequence in $L^{p'}(\R)$, 
we cannot rule out that 
cancellations may occur on a given subset $E_k \sbt \R$.

In order to overcome this obstacle, we will now modify the functions $h_k^*$. 
Suppose that we are given a sequence of scalars
$\{b_{ts}\}$, $(t,s)\in\Lambda$, with only finitely many 
nonzero terms. For any sequence $\theta = \{\theta_{ts}\}$
of signs, i.e.\ such that $\theta_{ts}=\pm 1$, put
\begin{equation}
\label{eq:hk.s.def}
h_{k,\theta}^*  = \sum_{(t,s)\in\Lambda_k} \theta_{ts}b_{ts}g^*_{ts},
\quad  k \in \Z.
\end{equation}
We observe that the estimate \eqref{eq:keyest_dual} 
is still valid for the function $h_{k,\theta}^*$. 

Our goal now is to show that \emph{there exists some choice of signs}  $\theta = \{\theta_{ts}\}$
(which depends on the scalars $\{b_{ts}\}$)
such that the functions $\{h_{k,\theta}^*\}$, $k \in \Z$,
after a suitable normalization, satisfy
the conditions of \lemref{lem:JO.EXT}. 

To this end,  we put
\begin{equation}
\label{eq:epschoice}
\eps = (4 L K^2 )^{-1} A_p,
\end{equation} 
and we choose a large interval $I\subset \R$ such that 
\begin{equation}
\label{eq:i.c}
\sup_{0 \le t < \del} \| \tau_t g \|_{L^p(\R \setminus I)} 
< \eta, \quad \eta = C^{-1}_p \eps.
\end{equation}	
Let $E_k = I+k\delta$, $k \in \Z$. It then follows from
\eqref{eq:5.1} and  \eqref{eq:i.c}   that
\begin{equation}
\label{eq:i.t.g}
\|\tau_t g\|_{L^p(\R\setminus E_k)} < \eta, \quad
(t,s)\in \Lambda_k.
\end{equation}	
We also observe that the intervals 
$E_k$ satisfy $\sum_k \1_{E_k}\le M$ a.e., for a 
suitable constant $M$  which depends only on $p$, $K$, $L$, $\del$,
 and the function $g$ (that is, $M$ does not
 depend on the scalars $\{b_{ts}\}$ or the signs $\{\theta_{ts}\}$).

Now, consider the functions
\begin{equation}
\label{eq:hk.t.bts}
h_{k,\theta} =  \sum_{(t,s)\in\Lambda_k}  \theta_{ts}\overline{b}_{ts} e_s \ts_t g, 
\quad  k \in \Z.
\end{equation}
We will show   that there exists a choice of signs $\{\theta_{ts}\}$ such that ``most part'' of the $L^{p}$ mass of the function $h_{k,\theta}$ is concentrated on the interval $E_k$. To this end, 
we put a Rademacher sequence $\eps_{ts}$ instead of $\theta_{ts}$, and take expectation. Using type 2 inequality once again (\lemref{lem:type}) and \eqref{eq:i.t.g}, we get 
\begin{align}
\mathbb{E} \Big\| \sum_{(t,s)\in\Lambda_k}  
\eps_{ts}\overline{b}_{ts} e_s \ts_t g \Big\|_{L^p(\R\setminus E_k)}
& \le C_p  \Big( \sum_{(t,s)\in\Lambda_k}
|b_{t,s}|^2 \|\tau_t g\|_{L^p(\R\setminus E_k)}^2 \Big)^{1/2}\\
& \le C_p \eta \Big( \sum_{(t,s)\in\Lambda_k}|b_{t,s}|^2 \Big)^{1/2}.
\end{align}
Hence, for each $k \in \Z$ we can fix signs 
$\{ \theta_{t,s} \}$, $(t,s)\in\Lambda_k$,
which depend on $\{b_{ts}\}$,
such that for this choice
of signs the function $h_{k,\theta}$
given by \eqref{eq:hk.t.bts} satisfies the inequality
\begin{equation}\label{eq:conc}
\|h_{k,\theta}|_{\R\setminus E_k}\|_p 
\le C_p \eta \Big( \sum_{(t,s)\in\Lambda_k}|b_{t,s}|^2 \Big)^{1/2} 
= \eps \Big( \sum_{(t,s)\in\Lambda_k}|b_{t,s}|^2 \Big)^{1/2}.
\end{equation}

Next, we show that for the same choice of   signs 
$\theta = \{ \theta_{ts} \}$,  a ``substantial part''   
of the $L^{p'}$ mass of the function 
$h_{k,\theta}^*$ given by \eqref{eq:hk.s.def}
 is concentrated on $E_k$.
Since $\{\pphi_{ts}^*\} \sbt X^*$ is the
system of functionals biortho\-gonal to the system 
$\{e_s\tau_t g\} \sbt X$, it follows from \eqref{eq:gts.s.d}
that the system $\{ g_{ts}^* \} \sbt L^{p'}(\R)$ is again biorthogonal to $\{e_s\tau_t g\}$. 
This implies that
\begin{equation}
\label{eq:biort}
\int_\R h_{k,\theta}\cdot h^*_{k,\theta} = \sum_{(t,s)\in\Lambda_k} |b_{ts}|^2.
\end{equation}
In turn, using \eqref{eq:keyest_dual}, \eqref{eq:conc}, \eqref{eq:biort} we conclude that 
\begin{align}
 \Big| \int_{E_k} h_{k,\theta}\cdot h^*_{k,\theta}\Big| 
& =  \Big|\int_{\R} h_{k,\theta}\cdot h^*_{k,\theta} -\int_{\R\setminus E_k} h_{k,\theta}\cdot h^*_{k,\theta}\Big| \\
& \ge \sum_{(t,s)\in\Lambda_k} |b_{ts}|^2 -\|h_{k,\theta}|_{\R\setminus E_k}\|_p\cdot \|h^*_{k,\theta}\|_{p'}\\
& \ge    ( 1- 2 \eps L K^2  A_p^{-1}  )
\sum_{(t,s)\in\Lambda_k} |b_{ts}|^2 
= (1/2) \sum_{(t,s)\in\Lambda_k} |b_{ts}|^2, \label{eq:hhs.eq}
\end{align}
where in  the last equality  we used our choice \eqref{eq:epschoice}
of the number $\eps$ . It remains to apply the latter estimate
together with  H\"{o}lder's inequality and \eqref{eq:5.17}, and we get
\begin{equation}
\label{eq:dualconc}
\|h^*_{k,\theta}|_{E_k}\|_{p'} 
 \ge \|h_{k,\theta}\|_p^{-1}\cdot \Big| \int_{E_k} h_{k,\theta}\cdot h^*_{k,\theta}\Big|
 \ge (2KC_p)^{-1} \Big(\sum_{(t,s)\in\Lambda_k} |b_{ts}|^2\Big)^{1/2}.
\end{equation}

The concentration inequality \eqref{eq:dualconc} now provides us with the possibility of applying
\lemref{lem:JO.EXT} in the space $L^{p'}(\R)$. This will be done in the next (and final) step.

\subsubsection{Step 5: The end of the proof}  
For each $k \in \Z$, we put 
\begin{equation}
\label{eq:fk.def}
f_{k,\theta} = \Big(\sum_{(t,s)\in\Lambda_k} |b_{ts}|^2\Big)^{-1/2} h_{k,\theta}^*.
\end{equation}
Our goal is to apply \lemref{lem:JO.EXT} to the system 
$\{ f_{k,\theta} \}$, $k \in \Z$,  in the space $L^{p'}(\R)$.

Since the system $\{ g_{ts}^* \}$, $(t,s) \in\Lambda$, forms an $(LK)$-unconditional 
basic sequence in $L^{p'}(\R)$  (due to  \lemref{lem:dual.gts.s}),
 and since according to \eqref{eq:hk.s.def}, \eqref{eq:fk.def},
 each function $f_{k,\theta}$ is a linear combination of functions 
$\{ g_{ts}^* \}$, $(t,s)\in\Lambda_k$, and the ``blocks''
$\Lam_k$ are disjoint,  then also the system 
$\{ f_{k,\theta} \}$, $k \in \Z$,  forms an $(LK)$-unconditional basic sequence
in $L^{p'}(\R)$.
  
The sequence of functions $\{f_{k,\theta}\}$, $k\in\Z$, is bounded in $L^{p'}(\R)$, since we have
\begin{equation}
\|f_{k,\theta}\|_{p'} \le 2LK^2 A_p^{-1}
\end{equation}
according to \eqref{eq:keyest_dual}.
Moreover,  it follows from the estimate \eqref{eq:dualconc} that 
\begin{equation}
\|f_{k,\theta}|_{E_k}\|_{p'}\ge (2KC_p)^{-1},
\end{equation}
that is, the part of the $L^{p'}$ mass of the function $f_{k,\theta}$ which
is concentrated on the interval $E_k$ is bounded from below by
a positive constant that depends on $p$ and $K$ only.

We also recall that the intervals 
$\{ E_k \}$, $k \in \Z$,  satisfy the condition $\sum_k \1_{E_k}\le M$ a.e., 
where $M$ is a constant  which depends only on $p$, $K$, $L$, $\del$, and the function $g$.

This  means that the sequence $\{ f_{k,\theta} \}$, together 
with the sets $\{E_k\}$, satisfies the conditions of 
\lemref{lem:JO.EXT} in the space $L^{p'}(\R)$, and we have
$1 < p' \le 2$. Hence, by the lemma, the  sequence 
 $\{ f_{k,\theta} \}$ is equivalent to the unit vector
 basis of  $\ell^{p'}$, and the implied constants
 in this equivalence depend 
neither on the scalars $\{b_{ts}\}$ nor on the 
sequence of signs $\theta = \{\theta_{ts}\}$.

In turn, this allows us to conclude that the system $\{g_{ts}^*\} \sbt L^{p'}(\R)$
 is equivalent to the standard unit vector basis of $(\bigoplus_{k\in\Z} \ell^2_{\Lambda_k})_{p'}$. Indeed, using the unconditionality of this system, we have
\begin{align}
&\Big\| \sum_{k\in\Z} \sum_{(t,s)\in\Lambda_k} b_{ts} g^*_{ts} \Big\|_{p'}
\asymp \Big\| \sum_{k\in\Z} \sum_{(t,s)\in\Lambda_k} \theta_{ts}b_{ts} g^*_{ts} \Big\|_{p'} 
= \Big\| \sum_{k\in\Z} h_{k,\theta}^* \Big\|_{p'} \label{eq:equiv.gts.1} \\
& \qquad = \Big\|\sum_{k\in\Z}  \Big(\sum_{(t,s)\in\Lambda_k}
 |b_{ts}|^2\Big)^{1/2}  f_{k, \theta} \Big\|_{p'} 
 \asymp \Big(\sum_{k\in\Z}\Big( \sum_{(t,s)\in\Lambda_k}
 |b_{ts}|^2 \Big)^{p'/2} \Big)^{1/p'}. \label{eq:equiv.gts.2}
\end{align}

Lastly, we observe that \eqref{eq:equiv.gts.1}, \eqref{eq:equiv.gts.2}
imply the equivalence \eqref{eq:toprove_bas_dual},
due to  the definition \eqref{eq:gts.s.d} of the system $\{ g_{ts}^* \}$
and using the inequalities  \eqref{eq:dual_proj}.
Hence, the equivalence \eqref{eq:toprove_bas_dual} is finally
established, and this completes the proof of \thmref{thm:uncseq_p>2}.
\qed

\subsection{Example}
We now present an example showing that \thmref{thm:uncseq_p>2} becomes false if we drop the assumption that $\G(g; \Lambda)$
spans a complemented subspace of $L^p(\R)$.

The construction is based on \cite[Example 2.16]{OSSZ11} which shows that there is a function $g \in L^p(\R)$, $p>2$, such that the system of translates $\{g(x-j)\}_{j=1}^{\infty}$ forms an unconditional basic sequence in  the space $L^p(\R)$, but the system is not equivalent to the standard unit vector basis of $\ell^p$. We put this example in a more general form, and show that it moreover satisfies the  stronger assertion in the following theorem.

\begin{thm}
\label{thm:n.c.s}
Given $p>2$, there is a function $g \in L^p(\R)$ such that the system of translates $\{g(x-j)\}_{j=1}^{\infty}$ forms an unconditional basic sequence in $L^p(\R)$, but the system is not equivalent to any sub\-sequence of the standard unit vector basis of $\ell^p(\ell^2)$.
\end{thm}

\begin{proof}
For arbitrary complex scalars $\{c_k\}_{k=0}^\infty$ such that $\sum_{k=0}^\infty |c_k|^p =1$, define
\begin{equation}
g (x) = \sum_{k=0}^\infty c_k e_{2^k}(x) \1_{[k, k+1]}(x), \quad x \in \R.
\end{equation}
Then $\|g\|_p^p = \sum_{k=0}^\infty |c_k|^p = 1$.
We consider the system $\{g(x-j)\}_{j=1}^{\infty}$ in $L^p(\R)$.

If  $a=\{a_j\}_{j=1}^\infty$ is a sequence of scalars with only finitely many nonzero terms, denote
\begin{equation}
\Phi_a (x) = \sum_{j=1}^\infty a_j g(x-j).
\end{equation}
We observe that $\Phi_a$ is supported on $[1, \infty)$, and for $l \ge 1$, on the interval $[l,l+1]$ we have
$\Phi_a  = \sum_{k=0}^{l-1} a_{l-k} c_k e_{2^k}$. By an application of \lemref{lem:lacun}, we conclude that
\begin{equation}
\label{eq:phi:l}
\|\Phi_a\|_p^p \asymp
\sum_{l=1}^{\infty}
\Big( \sum_{k=0}^{l-1} |a_{l-k}|^2 |c_k|^2 \Big)^{p/2},
\end{equation}
with the implied constants depending only on $p$.

In particular, \eqref{eq:phi:l} implies (e.g.\ using \cite[Propositions 1.1.9 and 3.1.3]{AK16}) that the system of translates $\{g(x-j)\}_{j=1}^{\infty}$ forms an unconditional basic sequence in $L^p(\R)$.

We now show that, for a suitable choice of the sequence $\{c_k\}_{k=0}^{\infty}$, the system of translates $\{g(x-j)\}_{j=1}^{\infty}$ is not equivalent to any subsequence of the standard unit vector basis of $\ell^p(\ell^2)$. We choose the sequence $\{c_k\}_{k=0}^{\infty}$ so that
$\sum_{k=0}^\infty |c_k|^2 = +\infty$. It follows from this condition that
\begin{equation}
\label{eq:ck:l}
\sum_{l=1}^{n} \Big( \sum_{k=0}^{l-1}
|c_k|^2 \Big)^{p/2} \gg n, \quad
n \to + \infty.
\end{equation}

Assume now that the system  $\{g(x-j)\}_{j=1}^{\infty}$ is equivalent to some subsequence of the standard unit vector basis of $\ell^p(\ell^2)$. We will show that this leads to a contradiction. The assumption means that the set of positive integers $\{1,2,3,\dots\}$ may be partitioned into disjoint sets $\{J_s\}$ (where each $J_s$ is a finite or countable set) such that
\begin{equation}
\label{eq:phi:x}
\|\Phi_a\|_p \asymp \Big( \sum_{s} \Big(\sum_{j \in J_s} |a_j|^2\Big)^{p/2} \Big)^{1/p}.
\end{equation}

If we have $\sup_s |J_s| < + \infty$, 
i.e.\ there are boundedly many elements in each set $J_s$, then the right hand side of \eqref{eq:phi:x} is equivalent to the norm $\|a\|_{\ell^p}$. But in this case, if we choose the sequence $a=(1,1,\dots,1,0,0,\dots)$ such that the number of $1$'s is $n$, then we have $\|a\|^p_{\ell^p} = n$, while
the right hand side of \eqref{eq:phi:l}
becomes at least as large as the quantity on the left hand side of \eqref{eq:ck:l} which grows faster than $n$. This contradicts \eqref{eq:phi:x}, so it is not possible that $\sup_s |J_s| < + \infty$.

Hence, there must be sets $J_s$ of arbitrarily large cardinality. We now use the following observation:\footnote{Let $g \in L^p(\R)$, $\|g\|_p=1$, and let $n$ be a positive integer. We choose a small $\eps = \eps(n,p) > 0$ and decompose $g = \pphi + \psi$ where $\pphi$ has compact support, $\|\pphi\|_p \le 1$, and $\|\psi\|_p < \eps$. Assume now that a sequence of real numbers $\{t_j\}_{j=1}^{n} $ is $M$-separated, where $M$ is the length of an interval containing the support of $\pphi$. Then $\| \sum_{j=1}^{n} g(x-t_j) \|_p \le n^{1/p}+ \eps n < 2n^{1/p}$ for $\eps$  small enough.} 
given a positive integer $n$, one can find a positive integer $M=M(g,n,p)$ sufficiently large, such that if a sequence of real numbers $\{t_j\}_{j=1}^{n} $ is $M$-separated, i.e.\ satisfies $|t_i-t_j| \ge M$, $ i \neq j$, then we have
$\| \sum_{j=1}^{n} g(x-t_j) \|_p < 2n^{1/p}$. We then choose some $s_0 = s_0(n,M)$ such that $|J_{s_0}| \ge (n-1)M + 1$. This condition ensures that we can find a subset $J \sbt J_{s_0}$ which is $M$-separated and contains exactly $n$ elements. As a consequence, if we define a  sequence
$a = \{a_j\}_{j=1}^{\infty}$ by
$a_j = 1$ for $j \in J$, and $a_j = 0$ otherwise, then we obtain $\|\Phi_a\|_p < 2n^{1/p}$. But on the other hand, since the support $J$ of the sequence $a$ is entirely contained in $J_{s_0}$, the right hand side of \eqref{eq:phi:x} is equal to the norm $\|a\|_{\ell^2} = n^{1/2}$. This contradicts  \eqref{eq:phi:x} for $n$ sufficiently large.

This shows that indeed, the system 
$\{g(x-j)\}_{j=1}^{\infty}$ is not equivalent to any subsequence of the standard unit vector basis of $\ell^p(\ell^2)$, and concludes the proof of \thmref{thm:n.c.s}.
\end{proof}

% =======================================

\section{The case $p > 2$: Unconditional Gabor frames}
\label{sec:framesp>2}

In this section we prove \thmref{thm:framesp>2}. There are two parts to this theorem. 

First, we construct an unconditional Gabor frame in  $L^p(\R)$, $p>2$,
for an arbitrary time--frequency shift set $\Lambda$ that is
not contained in any vertical strip $[-a, a]\times\R$.  The proof is 
done by a minor modification of the construction 
in \cite{FOSZ14} of unconditional Schauder frames of translates.

The bulk of the section is devoted to the second part of 
\thmref{thm:framesp>2}, namely, we prove that 
a Gabor system with time--frequency shift set $\Lambda$ 
contained in some vertical strip $[-a, a]\times\R$
cannot form an unconditional Schauder frame in $L^p(\R)$, $p \ge 2$.

\subsection{Proof of part (i) of \thmref{thm:framesp>2}}
Let $p>2$. Our goal is to show that if the set $\Lam \sbt \R \times \R$ is not
 contained in any vertical strip $[-a, a]\times\R$,
 then there is a function $g\in L^p(\R)$ and a 
sequence $\{ g^*_{ts} \}$ in $(L^p(\R))^*$,
such that the system $\{(e_s \tau_t g, g^*_{ts})\}$, $(t,s) \in \Lam$, 
forms an unconditional Schauder frame in the space $L^p(\R)$.

The construction is done by a modification of \cite[Theorem 3.2]{FOSZ14}.

First, observe that  by \lemref{lem:approxframe}
it would suffice to construct a function $g\in L^p(\R)$ and a 
sequence $\{ g^*_{ts} \}$ in $(L^p(\R))^*$,
such that the system $\{(e_s \tau_t g, g^*_{ts})\}$, $(t,s) \in \Lam$, 
forms an unconditional \emph{approximate} Schauder frame in $L^p(\R)$.

We consider the system of Haar functions $\{h_k\}_{k=1}^\infty$ 
on $\R$, normalized in $L^p(\R)$ (see \cite[Section 6.1]{AK16} 
for the definition and basic properties of the Haar functions on the 
segment $[0,1]$; in our case, we consider also the copies of this 
system on every interval $[n, n+1]$, $n\in\Z$, and we note that
the specific enumeration of the resulting system is not important
for our purpose). The system 
$\{h_k\}_{k=1}^\infty$ forms a $K_p$-unconditional basis in the 
space $L^p(\R)$ for a suitable constant $K_p$ 
(in fact, one can take $K_p = p-1$,
see \cite[Theorem 6.1.7]{AK16}). We denote the system of
biorthogonal functionals by $\{h_k^*\}_{k=1}^{\infty}$.

Next, we choose a sequence of positive integers 
$\{N_k\}_{k=1}^\infty$ satisfying the condition 
$\sum_{k=1}^\infty N_k^{1-p/2} < (2K_p)^{-p/2}$.
We then partition the set of positive integers 
$\{1,2,3,\dots\}$ into consecutive ``blocks'' 
$\{J_k\}_{k=1}^\infty$ such that 
$|J_k| = N_k$.

Since the set $\Lam$ is not
 contained in any vertical strip $[-a, a]\times\R$,
there exists a sequence of points $(t_i, s_i) \in \Lam$
 such that $|t_i| \to +\infty$. 
By passing to a subsequence, we may assume that the absolute values
$|t_i|$ are increasing and grow arbitrarily fast.  

Let $\mathcal{Q}$ be the set of all quadruples $(i, k, j, l)$ 
such that $i \in J_k$, $j \in J_l$,  $i \neq j$. 
For each quadruple $(i, k, j, l) \in \mathcal{Q}$, we define the set
\begin{equation}
\label{eq:6.1}
E(i,k,j,l) = \supp(h_k) + t_j - t_i.
\end{equation}
We make the following observation:
by choosing the sequence of absolute values
$|t_i|$ growing sufficiently fast, 
and using the fact that $\supp(h_k)$ is contained
in an interval of length at most $1$,
one can ensure that the sets
$\{E(i,k,j,l) \}$, $ (i,k,j,l) \in \mathcal{Q}$,
are pairwise disjoint.
This claim can be checked directly and the
details are left to the reader.

Now, we define our function $g$ as
\begin{equation}\label{eq:6.2}
g = \sum_{k=1}^\infty \sum_{i \in J_k} N_k^{-1/2}  \ts_{-t_i} (e_{-s_i} h_k).
\end{equation}
This differs from the construction in \cite[Theorem 3.2]{FOSZ14} only by the introduction of the exponential factors $e_{-s_i}$. By using the fact that
 the sets $ E(i,k,j,l) $ and $E(i',k',j,l)$ are disjoint 
whenever $(i,k,j,l)$ and  $(i',k',j,l)$ are  two quadruples in $\mathcal{Q}$
with $i \ne i'$, it follows that the summands in  \eqref{eq:6.2} have
pairwise disjoint supports. Hence
\begin{equation}
\|g\|^p_p = \sum_{k=1}^\infty N_k^{-p/2}\sum_{i \in J_k}\|\ts_{-t_i}  (e_{-s_i} h_k)\|_p^p 
= \sum_{k=1}^\infty N_k^{1-p/2} < \infty,
\end{equation}
that is, the function $g$ belongs to $L^p(\R)$.

Next, we define the coordinate functionals. 
For $j \in J_l$ we put $g^*_{t_j s_j} = N_l^{-1/2} h_l^*$,
while for all other points $(t,s)\in\Lambda$ we put $g^*_{ts} = 0$.
It can then be checked in the same way as in \cite[Theorem 3.2]{FOSZ14} that 
the system $\{(e_s \ts_t g, g^*_{ts})\}$, $(t,s)\in\Lambda$, forms
an unconditional approximate Schauder frame in the space $L^p(\R)$.
The main idea is to 
consider an arbitrary function $f \in L^p(\R)$, and split the series
\begin{equation}
\label{eq:6.m.s}
\sum_{j=1}^\infty g^*_{t_j s_j}(f) e_{s_j} \ts_{t_j} g 
= \sum_{k=1}^\infty \sum_{i \in J_k} 
 \sum_{l=1}^\infty \sum_{j \in J_l} 
 N_k^{-1/2}  N_l^{-1/2} h^*_l(f) e_{s_j} \ts_{t_j}  \ts_{-t_i} (e_{-s_i} h_k)
\end{equation}
into a main term and an error term as
\begin{equation}
\label{eq:6.m.e}
 \sum_{k=1}^\infty \sum_{i \in J_k} N_k^{-1} h^*_k(f) h_k
+ \sum_{(i,k,j,l)  \in \mathcal{Q}} N_k^{-1/2} N_l^{-1/2}   h^*_l(f) e_{s_j} \ts_{t_j}  \ts_{-t_i} (e_{-s_i} h_k).
\end{equation}
The first sum in \eqref{eq:6.m.e} converges unconditionally to $f$, since  the Haar system
$\{h_k\}_{k=1}^\infty$ forms an unconditional basis in the space $L^p(\R)$.
The summands in the second sum of
\eqref{eq:6.m.e} have pairwise disjoint supports, 
since the sets $\{E(i,k,j,l) \}$, $ (i,k,j,l) \in \mathcal{Q}$,
are pairwise disjoint. Hence,
the $L^p(\R)$ norm of the error term does not exceed
\begin{equation}
\label{eq:6.r.p}
\Big( \sum_{k=1}^\infty \sum_{i \in J_k} 
 \sum_{l=1}^\infty \sum_{j \in J_l} 
 N_k^{-p/2}  N_l^{-p/2} |h^*_l(f)|^p \Big)^{1/p} 
 \le  K_p  \|f\|_p \Big( \sum_{k=1}^\infty N_k^{1-p/2} \Big)^{2/p}.
\end{equation}
Since the constant
$ K_p \big( \sum_{k=1}^\infty N_k^{1-p/2} \big)^{2/p}$
is strictly smaller than $1$, this implies that
the series \eqref{eq:6.m.s} defines a bounded and 
invertible linear operator on $L^p(\R)$.
We refer the reader to \cite[Theorem 3.2]{FOSZ14} for the
complete details and rigorous justification of the claim
that the system $\{(e_s \ts_t g, g^*_{ts})\}$, $(t,s)\in\Lambda$, forms
an unconditional approximate Schauder frame in  $L^p(\R)$,
and we note that the exponential factors introduced in our
proof  do not affect any  of the estimates in \cite{FOSZ14}.
Our result is thus proved.
\qed

\subsection{Proof of part (ii) of \thmref{thm:framesp>2}}
We now turn to prove the main part of our theorem.
Let $p \ge 2$, and let $\Lam \sbt \R \times \R$ be a countable set which is
contained in some vertical strip $[-a, a]\times\R$.
Our goal is to show that there do not exist
any function $g\in L^p(\R)$ and 
coefficient functionals $\{ g^*_{ts} \}$ in $(L^p(\R))^*$,
such that the system $\{(e_s \tau_t g, g^*_{ts})\}$, $(t,s) \in \Lam$, 
forms an unconditional Schauder frame in  $L^p(\R)$.

By adding extra elements to the system with zeros as coefficient functionals,
we may assume that $\Lambda = \TTT \times \SSS$ 
where $\TTT, \SSS \sbt \R$ are finite or countable sets, and 
\begin{equation}
\label{eq:t.a}
\TTT\subset [-a,a].
\end{equation}
 Suppose to the contrary that 
$\{(e_s \ts_t g, g_{ts}^*)\}$, $(t,s) \in  \TTT \times \SSS$, forms
a $K$-unconditional Schauder frame in the space $L^p(\R)$. It means that any $f\in L^p(\R)$
can be expanded in an unconditionally convergent series 
\begin{equation}
\label{eq:expansion_p>2}
f = \sum_{t\in \TTT} \sum_{s\in \SSS} g_{ts}^*(f) e_s \ts_t g.
\end{equation}
We need to show that this assumption leads to a contradiction.

Our approach consists of two main steps. In the first step,
 we will prove an estimate on the coefficients $\{g_{ts}^*(f)\}$,
 $(t,s) \in \TTT \times \SSS$,
similar to the case $1 < p < 2$. However, in the case
 $p > 2$, the space $L^p(\R)$ does not satisfy cotype 2 inequality,
and we obtain the required estimate using
Lemmas \ref{lem:squarefunc} and \ref{lem:mink}.

In the second step, we will show that for any Schwartz function 
$\pphi$ whose Fourier transform $\ft{\pphi}$ has compact
support, the mapping $f\mapsto f\ast\pphi$ defines a 
compact operator   $L^p(\R) \to L^p(\R)$. 
This will lead  to the desired contradiction.

We note that in the case $p=2$,
our proof could be significantly simplified. 
By a development of the approach from \cite{LT25a} 
and an application of Plancherel's theorem, one can show
that if \eqref{eq:t.a} holds and the system
$\{(e_s \ts_t g, g_{ts}^*)\}$, $(t,s) \in  \TTT \times \SSS$, forms
an unconditional Schauder frame in the space $L^2(\R)$,
then for any bounded interval $I \sbt \R$,
the mapping $f \mapsto f \ast \ft{\1}_I$ 
defines a compact operator $L^2(\R) \to L^2(\R)$, 
which gives the desired contradiction.

Our proof in the general case can be viewed as a certain 
adaptation of these ideas to $L^p(\R)$ spaces, $p > 2$,
where Plancherel's theorem is no longer available.
In this case, our proof uses Rubio de Francia's Littlewood--Paley inequality
for arbitrary intervals, which was proved in  \cite{Rub85}.

\subsection{The estimate on the coefficients}
Without any loss of generality, we may assume that $\|g\|_p = 1$. 
We can then fix a big interval $I\subset \R$ such that
\begin{equation}
\label{eq:interval}
\inf_{|t| \le a + 1} \| \ts_t g\|_{L^p(I)} \ge 3/4.
\end{equation}

Next, we choose a small $0 <\delta < 1$ (depending on $g$ and $I$) and put 
\begin{equation}
\label{eq:tn.sk.1}
\TTT_n = \TTT\cap [n\delta, (n+1)\delta), \quad
\SSS_k = \SSS\cap [k\delta, (k+1)\delta), \quad  n,k \in \Z.
\end{equation}
We observe that due to \eqref{eq:t.a},
 only a finite number of sets $\TTT_n$ are nonempty.

\begin{lem}
\label{lem:d.es.tt}
If $\delta > 0$ is chosen sufficiently small (depending on $g$ and $I$) then 
\begin{equation}
\label{eq:6.6}
\|e_s \ts_t g - e_{k\delta} \ts_{n\delta} g\|_{L^p(I)} \le 1/4,
\quad (t,s) \in \TTT_n \times \SSS_k.
\end{equation}
\end{lem}

\begin{proof}
The triangle inequality yields
\begin{equation}
\label{eq:d.es.tt.1}
\|e_s \ts_t g - e_{k\delta} \ts_{n\delta} g\|_{L^p(I)}
\le \|(e_s-e_{k\delta}) \ts_t g\|_{L^p(I)} 
+ \|e_{k\delta} (\ts_t g - \ts_{n\delta}g)\|_{L^p(I)}.
\end{equation}
The first summand on the right hand side of \eqref{eq:d.es.tt.1} does not exceed 
\begin{equation}
\|e_s-e_{k\delta}\|_{L^\infty(I)}\cdot \|\ts_t g\|_{L^p(I)}
\le 	\|e_s-e_{k\delta}\|_{L^\infty(I)} \le 1/8
\end{equation}
for $\delta $  small enough, 
because $|s-k\delta| < \delta$. The second summand equals 
\begin{equation}
\|\ts_t g - \ts_{n\delta}g\|_{L^p(I)}
\le \|\ts_t g - \ts_{n\delta}g\|_{L^p(\R)}  \le 1/8
\end{equation}
provided that
 $\delta $  is small enough, this time
because $|t-n\delta| < \delta$, and \eqref{eq:6.6} follows.
\end{proof}

\emph{Remark}. Of course, by taking a 
sufficiently small $\delta $ we could have achieved any 
positive constant on the right hand side of 
the inequality in \eqref{eq:6.6}.

We now fix the number $\delta $ provided by 
\lemref{lem:d.es.tt}, and assume that $0 < \delta < 1$.
In turn, this choice of $\del$ determines the sets
$\TTT_n$ and $\SSS_k$ defined by \eqref{eq:tn.sk.1}.

Suppose that 
 $n \in \Z$ is such that the set $\TTT_n$ is nonempty.
By our assumption, the system
$\{(e_s \ts_t g, g_{ts}^*)\}$, $(t,s) \in  \TTT \times \SSS$, forms
a $K$-unconditional Schauder frame in $L^p(\R)$.
Hence, for any $f\in L^p(\R)$, 
 and for arbitrary signs $\theta_k = \pm 1$, we have
\begin{equation}
\Big\| \sum_{k\in\Z} \theta_k   \sum_{t\in \TTT_n} \sum_{s\in \SSS_k}
 |g_{ts}^*(f)| e_s \ts_t g \Big\|^p_{p}\le K^p \|f\|^p_p.
\end{equation}
Putting a Rademacher sequence into the signs $\theta_k$, 
taking expectation and applying \lemref{lem:squarefunc}, we get
\begin{equation}
\Big( \int_\R \Big( \sum_{k\in\Z} \Big|  \sum_{t\in \TTT_n} \sum_{s\in \SSS_k}
 |g_{ts}^*(f)|  (e_s \ts_t g)(x)  \Big|^2 \Big)^{p/2} dx \Big)^{1/p}\le K A_p^{-1}  \|f\|_p.
\end{equation}
Now, we change the domain of integration from $\R$ to the interval $I$,
 and rewrite the estimate in a more compact form as
\begin{equation}
\label{eq:6.12}
\Big\| \Big(  \sum_{t\in \TTT_n} \sum_{s\in \SSS_k}
 |g_{ts}^*(f)|e_s \ts_t g \Big)_k \Big\|_{L^p(I;\ell^2)}\le  K A_p^{-1} \|f\|_p.
\end{equation}
By using the triangle inequality in the space $L^p(I;\ell^2)$, 
we can estimate the left hand side of 
\eqref{eq:6.12} from below as 
\begin{align}
&\Big\|   \Big(   \sum_{t\in \TTT_n} \sum_{s\in \SSS_k}
|g_{ts}^*(f)|e_s \ts_t g \Big)_k  \Big\|_{L^p(I;\ell^2)} 
 \ge \Big\| \Big(  \sum_{t\in \TTT_n}  \sum_{s\in \SSS_k}
  |g_{ts}^*(f)|e_{k\delta} \ts_{n\delta}g \Big)_k \Big\|_{L^p(I;\ell^2)} \label{eq:6.13} \\ 
 & \qquad  \qquad  - \Big\| \Big(   \sum_{t\in \TTT_n}  \sum_{s\in \SSS_k} |g_{ts}^*(f)|
   (e_{k\delta} \ts_{n\delta}g -e_s \ts_t g )\Big)_k \Big\|_{L^p(I;\ell^2)}. \label{eq:6.14}
\end{align}

Since the set $\TTT_n$ is nonempty, 
we have $|n\delta| \le a+1$, due to
\eqref{eq:t.a}, \eqref{eq:tn.sk.1}.
So  we can use \eqref{eq:interval} to
estimate the right hand side of \eqref{eq:6.13} from below as
\begin{align}
&\Big\| \Big( \sum_{t\in \TTT_n}  \sum_{s\in \SSS_k}
 |g_{ts}^*(f)|e_{k\delta} \ts_{n\delta}g \Big)_k \Big\|_{L^p(I;\ell^2)}\\
&\qquad = \Big\|\Big(\sum_{k\in\Z} \Big(
\sum_{t\in \TTT_n}  \sum_{s\in \SSS_k}
|g_{ts}^*(f)| \Big)^2 \cdot |e_{k\delta} \ts_{n\delta} g|^2  \Big)^{1/2}\Big\|_{L^p(I)}\\
&\qquad =\| \ts_{n\delta} g\|_{L^p(I)}\cdot\Big( \sum_{k\in\Z} \Big(
\sum_{t\in \TTT_n}  \sum_{s\in \SSS_k}
|g_{ts}^*(f)| \Big)^2 \Big)^{1/2}\\
&\qquad \ge \frac{3}{4}\Big( \sum_{k\in\Z} \Big( 
\sum_{t\in \TTT_n}  \sum_{s\in \SSS_k}
|g_{ts}^*(f)| \Big)^2 \Big)^{1/2}.
\end{align}

Next, we  estimate the quantity in \eqref{eq:6.14} from above. To this end,
 we apply \lemref{lem:mink} and the triangle inequality in the space $L^p(I)$, and infer that 
\begin{align}
&\Big\| \Big( \sum_{t\in \TTT_n}  \sum_{s\in \SSS_k} 
 |g_{ts}^*(f)| (e_{k\delta} \ts_{n\delta}g -e_s \ts_t g )\Big)_k \Big\|_{L^p(I;\ell^2)}\\
&\qquad \le \Big( \sum_{k\in\Z} \Big\| \sum_{t\in \TTT_n}  \sum_{s\in \SSS_k} 
 |g_{ts}^*(f)| \cdot (e_{k\delta} \ts_{n\delta} g - e_s \ts_t g) \Big\|^2_{L^p(I)} \Big)^{1/2}\\
& \qquad \le \Big( \sum_{k\in\Z} \Big( \sum_{t\in \TTT_n}  \sum_{s\in \SSS_k} 
 |g_{ts}^*(f)|\cdot \|e_{k\delta} \ts_{n\delta}g - e_s \ts_t g\|_{L^p(I)} \Big)^2 \Big)^{1/2}. \label{eq:es.p.5}
\end{align}
Since we have chosen $\delta$ so that the inequality \eqref{eq:6.6} holds, we conclude that 
the  quantity in \eqref{eq:es.p.5} does not exceed
\begin{equation}
\frac{1}{4}\Big( \sum_{k\in\Z} \Big(
\sum_{t\in \TTT_n}  \sum_{s\in \SSS_k} 
|g_{ts}^*(f)| \Big)^2 \Big)^{1/2}.
\end{equation}

Collecting the above estimates, and using \eqref{eq:6.12},
 \eqref{eq:6.13}--\eqref{eq:6.14}, we get
\begin{equation}
\label{eq:es.p.8}
\Big( \sum_{k\in\Z} \Big( 
\sum_{t\in \TTT_n}  \sum_{s\in \SSS_k} 
|g_{ts}^*(f)| \Big)^2 \Big)^{1/2}\le 2K A_p^{-1} \|f\|_p.
\end{equation}
This estimate is valid for every $n \in \Z$ such that the set
$\TTT_n$ is non-empty. Since there are only finitely many
such values of $n$, due to 
\eqref{eq:t.a}, \eqref{eq:tn.sk.1},
we infer from \eqref{eq:es.p.8}
using the triangle inequality in the space $\ell^2$,
that for some positive constant $M$ 
(not depending  on $f$) the inequality
\begin{equation}
\label{eq:keyest}
\Big( \sum_{k\in\Z} \Big( \sum_{t\in \TTT} \sum_{s\in \SSS_k}
|g_{ts}^*(f)| \Big)^2 \Big)^{1/2}\le M \|f\|_p
\end{equation}
holds for every $f \in L^p(\R)$. 
This is our key estimate on the coefficients $\{g_{ts}^*(f)\}$.

\subsection{Convolution operator}
Let $\pphi$ be a Schwartz function on $\R$, and
\begin{equation}
\ft{\pphi}(\xi) =  \int_{\R} \pphi(x) e^{-2 \pi i \xi x} dx,
\quad  \xi \in \R,
\end{equation}
be the Fourier transform of $\pphi$, which is again
a Schwartz function on $\R$. We assume that
$\ft{\pphi}$ has compact support, but that
$\pphi$ is otherwise arbitrary. 
The function $\pphi$  induces a convolution operator $U_\pphi$ 
defined by $U_\pphi f = f\ast\pphi$,
which we consider as a bounded operator $L^p(\R) \to L^p(\R)$. 
Assuming that \eqref{eq:t.a} holds and the system
$\{(e_s \ts_t g, g_{ts}^*)\}$, $(t,s) \in  \TTT \times \SSS$, forms
an unconditional Schauder frame in $L^p(\R)$,
we will show that for any weakly null sequence 
$\{ f_n \}$ in $L^p(\R)$, we have
$\|U_\pphi f_n\|_p \to 0$.
(Since the space $L^p(\R)$ is reflexive, this means that 
$U_\pphi$ is a compact operator, see
\cite[Chapter VI, Proposition 3.3]{Con90},
but we will not use this fact).
This will clearly lead us to a contradiction, since we can e.g.\ take $\pphi$ 
nonnegative and not identically zero, 
and choose $f_n = \1_{[n,n+1]}$, and in this case 
we get that $U_\pphi f_n = \ts_n (U_\pphi f_0)$, and so 
$\| U_\pphi f_n \|_p$  does not tend to zero.

We therefore suppose that $\{ f_n \}$ is an arbitrary
 weakly null sequence in $L^p(\R)$. 
An application of the uniform boundedness principle yields that
 $\{ f_n \}$ must be a bounded sequence  in $L^p(\R)$, so we may assume
with no loss of generality that $\|f_n\|_p \le 1$. 
Since the system 
$\{(e_s \ts_t g, g_{ts}^*)\}$, $(t,s) \in  \TTT \times \SSS$, forms
an unconditional Schauder frame in $L^p(\R)$,
each $f_n$ admits an unconditionally convergent series expansion
\begin{equation}
\label{eq:fn.ucs}
f_n =  \sum_{k\in\Z} \sum_{t\in \TTT} \sum_{s\in \SSS_k}  
 g_{ts}^* (f_n)  e_s \ts_t g,
\end{equation}
where the convergence is in the $L^p(\R)$ norm.

Now, fix an arbitrarily small  $\eps > 0$, and
let $N$ be a positive integer. We write
\begin{equation}
\label{eq:def.an.bn.1} 
U_\pphi f_n  = f_n\ast\pphi   = A^N_n + B^N_n,
\end{equation}
where $A^N_n, B^N_n$ are defined by
\begin{equation}
 \label{eq:def.an.bn.2}
A^N_n = \Big(  \sum_{|k| \le N} \sum_{t\in \TTT}  \sum_{s\in \SSS_k}
 g_{ts}^*(f_n)  e_s \ts_t g \Big) \ast \pphi  
\end{equation}
and
\begin{equation}
 \label{eq:def.an.bn.3}
B^N_n =   \Big(  \sum_{|k| > N} \sum_{t\in \TTT}  \sum_{s\in \SSS_k}
 g_{ts}^*(f_n)  e_s \ts_t g \Big) \ast \pphi.
\end{equation}
We will estimate $A^N_n$ and $B^N_n$ separately. On one hand, we will 
show that if $N$ is sufficiently large (independently of $n$)
then  $ \|B^N_n\|_p \le \eps$.
 On the other hand, we will also show that if $N$ is fixed, then
for all sufficiently large $n$ we have $\|A^N_n\|_p \le\eps$.
Once we prove these two claims we are 
done: we first fix $N$ large enough, and then conclude 
from \eqref{eq:def.an.bn.1}  that 
for all sufficiently large $n$ we have $\|U_\pphi f_n\|_p \le 2\eps$.
Since $\eps $ is an arbitrarily small number,
 this shows that $\|U_\pphi f_n\|_p \to 0$.

It remains to prove the two required estimates of $A^N_n$ and $B^N_n$.

\subsection{The estimate of $A^N_n$}
First, let us assume that $N$ is fixed. We then need to 
show that for all sufficiently large $n$ we have
$\|A^N_n\|_p \le\eps$ . 
 By applying the inequality 
$\|\pphi \ast \psi \|_p \le \|\pphi\|_1 \|\psi\|_p$
to \eqref{eq:def.an.bn.2}, and using
 the triangle inequality  in  $L^p(\R)$, we get
\begin{equation}
\label{eq:an.c.1}
\|A^N_n \|_p \le  \|\pphi\|_1 \sum_{|k| \le N} \Big\|  \sum_{t\in\TTT} \sum_{s\in \SSS_k}
 g_{ts}^*(f_n) e_s \ts_t g \Big\|_p.
\end{equation}
Hence, it would be enough to show that for each $k \in \{-N, \dots, N\}$, the inequality
\begin{equation}
\label{eq:6.78.1}
\Big\|  \sum_{t\in\TTT} \sum_{s\in \SSS_k}  
g_{ts}^*(f_n) e_s \ts_t g \Big\|_p \le\frac{\eps}{2N+1}
\end{equation}
holds for all sufficiently large $n$.

To this end, we fix a large interval $J\subset\R$ such that
\begin{equation}
\label{eq:6.30}
\sup_{|t| \le a} \|\ts_t g\|_{L^p(\R\setminus J)} \le
\frac{\eps}{2M   (2N+1)},
\end{equation}
where $M$ is the constant from the inequality
\eqref{eq:keyest}. We then have
\begin{align}
&\Big\| \sum_{t\in\TTT}  \sum_{s\in \SSS_k}
 g_{ts}^*(f_n) e_s \ts_t g \Big\|_p 
\le \Big\|  \sum_{t\in\TTT} \sum_{s\in \SSS_k}
 g_{ts}^*(f_n) e_s \ts_t g \Big\|_{L^p(J)} \label{eq:6.77.1} \\
& \qquad  \qquad + \Big\|  \sum_{t\in\TTT} \sum_{s\in \SSS_k}
 g_{ts}^*(f_n) e_s \ts_t g \Big\|_{L^p(\R\setminus J)}. \label{eq:6.77.2}
\end{align}
It follows from \eqref{eq:t.a}, \eqref{eq:6.30} that  for every $n$, the 
quantity in \eqref{eq:6.77.2} does not exceed
\begin{align}
&  \sum_{t\in\TTT} \sum_{s\in\SSS_k} 
\|g_{ts}^*(f_n) e_s \ts_t g\|_{L^p(\R\setminus J)} 
=  \sum_{t\in\TTT} \sum_{s\in\SSS_k}
  |g_{ts}^*(f_n)|\cdot \|\ts_t g\|_{L^p(\R\setminus J)} \\
& \qquad \qquad \le  \frac{\eps}{2M (2N+1)} 
\Big( \sum_{t\in\TTT} \sum_{s\in \SSS_k} |g_{ts}^*(f_n)|\ \Big),
\end{align}
and, due to \eqref{eq:keyest} and since $\|f_n\|_p \le 1$,
 this is not greater than  $\frac1{2} (2N+1)^{-1} \eps$.

Hence, in order to establish the estimate
\eqref{eq:6.78.1}, it remains to show that
\begin{equation}
\label{eq:6.34}
\Big\| \sum_{t\in\TTT} \sum_{s\in \SSS_k} 
g_{ts}^*(f_n) e_s \ts_t g \Big\|_{L^p(J)}\le  \frac{\eps}{2(2N+1)}
\end{equation}
for all sufficiently large $n$. To show this, we will
use the assumption that $\{ f_n \}$ forms a weakly
null sequence in the space $L^p(\R)$. The main idea is that the double sum 
on the left hand side of \eqref{eq:6.34}, when
considered as an operator acting on $f_n$, 
can be arbitrarily well approximated by finite rank operators.

We now turn to prove the estimate \eqref{eq:6.34}.
We choose a small number $\eta > 0 $ and 
partition the (bounded) set $T \times S_k$ into a finite number
of pairwise joint nonempty subsets $\Omega_1, \dots, \Omega_\nu$ 
such that each $\Omega_i$ has diameter at most $\eta$.
For each $i = 1,\dots, \nu$, we then choose and fix an arbitrary point
$(t_i,s_i) \in \Omega_i$. Similarly to \lemref{lem:d.es.tt}, by
choosing $\eta  >  0 $ small enough  
(depending on $g$, $J$, $\eps$, $N$ and $M$)
we get
\begin{equation}\label{eq:6.36}
\| e_s \ts_t g  - e_{s_i} \ts_{t_i} g\|_{L^p(J)}
\le \frac{\eps}{4M (2N+1)}, \quad (t,s) \in \Omega_i.
\end{equation}

We now estimate the left hand side of \eqref{eq:6.34} as 
\begin{align}
& \Big\| \sum_{t\in\TTT} \sum_{s\in \SSS_k} 
 g_{ts}^*(f_n) e_s \ts_t g \Big\|_{L^p(J)}
\le \Big\| \sum_{i=1}^{\nu} \sum_{(t,s) \in \Omega_i}
 g_{ts}^*(f_n) e_{s_i} \ts_{t_i} g \Big\|_{L^p(J)} \label{eq:6.38.1} \\
& \qquad \qquad + \Big\|   \sum_{i=1}^{\nu} \sum_{(t,s) \in \Omega_i}
 g_{ts}^*(f_n) ( e_s \ts_t g -e_{s_i} \ts_{t_i} g) \Big\|_{L^p(J)}. \label{eq:6.38.2}
\end{align}
The quantity in \eqref{eq:6.38.2} is not greater than
\begin{align}
&  \sum_{i=1}^{\nu} \sum_{(t,s) \in \Omega_i}
 |g_{ts}^*(f_n)|\cdot\|e_s \ts_t g-e_{s_i} \ts_{t_i} g\|_{L^p(J)} \label{eq:6.39.1} \\
& \qquad  \qquad  
\le \frac{\eps}{4M(2N+1) } \sum_{t\in\TTT} \sum_{s\in\SSS_k}
 |g_{ts}^*(f_n)|\le \frac{\eps}{4(2N+1)}, \label{eq:6.39.2}
\end{align}
which follows from \eqref{eq:6.36} and our estimate \eqref{eq:keyest},
 using that $\|f_n\|_p \le 1$.

Next, we can write the term on the right hand side of \eqref{eq:6.38.1} as
\begin{equation}
\label{eq:6.77.6}
 \Big\| \sum_{i=1}^{\nu} \Big( \sum_{(t,s) \in \Omega_i}
 g_{ts}^*(f_n) \Big) e_{s_i} \ts_{t_i} g \Big\|_{L^p(J)}
 \le  \sum_{i=1}^{\nu} \Big| \sum_{(t,s) \in \Omega_i}
 g_{ts}^*(f_n)  \Big| = 
   \sum_{i=1}^{\nu} | h_{i}^*(f_n) |,
\end{equation}
where the inequality in \eqref{eq:6.77.6} holds
since $\|g\|_p = 1$, and where $h_i^*$ are defined by
\begin{equation}
\label{eq:6.41}
h_i^*(f) = \sum_{(t,s) \in \Omega_i}
 g_{ts}^*(f), \quad f \in L^p(\R).
\end{equation}
We observe that due to \eqref{eq:keyest}, the series in \eqref{eq:6.41}
converges absolutely and defines a
bounded linear functional $h_i^*$ on $L^p(\R)$. 
Since $\{f_n\}$ is a weakly null sequence, each summand
on the right hand side of   \eqref{eq:6.77.6} therefore
tends to zero as $n \to \infty$.
Since there are only  a finite number of these summands,
we conclude that
for all sufficiently large $n$ this sum is smaller than 
$\tfrac1{4} (2N+1)^{-1} \eps$. 

It remains to combine this estimate with
\eqref{eq:6.38.1}--\eqref{eq:6.38.2} and
 \eqref{eq:6.39.1}--\eqref{eq:6.39.2}
 in order to arrive at the desired inequality \eqref{eq:6.34}.
 The   estimate for $\| A^N_n \|_p$ is thus proved.

\subsection{The estimate of $B^N_n$}
It therefore remains to show that if $N$ is   sufficiently large 
(independently of $n$)
then $\|B^N_n\|_p \le \eps$. Here we will use our 
assumption that 
$\ft{\pphi}$ has compact support. The fact 
that $\{f_n\}$ is a weakly null sequence in $L^p(\R)$
will not play a role here;  the only property of $f_n$
that we will use is that $\|f_n\|_p\le 1$. 

We start by rewriting the expression \eqref{eq:def.an.bn.3} for $B^N_n$ as 
\begin{equation}
\label{eq:b2.d.1}
B^N_n(x) = \sum_{|k| > N} \sum_{t\in \TTT} \sum_{s\in \SSS_k}
g_{ts}^*(f_n)  \int_\R  g(y-t) e^{2\pi i s y} \pphi(x-y) dy,
\end{equation}
and we observe that the exchange of summation and integration in \eqref{eq:b2.d.1}
is justified using the 
unconditional convergence of the series \eqref{eq:fn.ucs}
in the $L^p(\R)$ norm, and the fact that $\pphi$ is a Schwartz function.

If $I \sbt \R$ is an interval, let $\Delta_I$
be the operator defined 
on Schwartz functions $f$ by
\begin{equation}
\label{eq:hm.def}
(\Delta_I  f)(x) = \int_{I}
\ft{f}(\xi) e^{2 \pi i \xi x} d \xi,
\end{equation}
that is, the ``Fourier partial sum'' operator associated with the interval $I$.
It is well known that $\Del_I$ extends uniquely to a bounded operator
$L^p(\R) \to L^p(\R)$, $1 <  p < \infty$;
this is a consequence
 of Riesz's theorem on the boundedness
of the Hilbert transform on $L^p(\R)$, 
see e.g.\ \cite[Chapter IV, Section 4.1, Theorem 4]{Ste70}.

We will use the following property of the operator
$\Delta_I$: Let  $f \in L^p(\R)$, $1 < p < \infty$, and let
$\psi$ be a Schwartz function with $\supp(\ft{\psi}) \sbt - I$.
Then for every $t \in \R$,
\begin{equation}
\label{eq:h.4}
\int_{\R} f  (y-t)  \psi(y)  dy = 
\int_{\R} (\Del_I f)  (y-t)  \psi(y)  dy.
\end{equation}
Indeed, \eqref{eq:h.4} follows from
\eqref{eq:hm.def} in the case where $f$ is a Schwartz function,
and it extends to all   $f \in L^p(\R)$
by the continuity of
the operator $\Del_I$ on $L^p(\R) $.

We now continue with our proof, again assuming that $p \ge 2$.
By our  assumption, the Fourier transform
 $\ft{\pphi}$ has compact support, so we may choose
$l = l (\pphi) > 0 $  such that $\supp( \ft{\pphi}) \sbt 
[- l +  1, l-1]$. Hence, the function $\psi_{s,x}(y) =  e^{2\pi i s y} \pphi(x-y)$
is a Schwartz function whose Fourier transform is supported
on $[s-l+1, s+l-1]$. In particular, if  $s \in \SSS_k$  then
\eqref{eq:tn.sk.1} implies that the latter
interval is contained in $ [k \del - l, k \del + l]$.
So  applying \eqref{eq:h.4} with $f = g$ and
$\psi = \psi_{s,x}$, we conclude from 
\eqref{eq:b2.d.1} that
\begin{equation}
\label{eq:b2.d.2}
B^N_n(x) = \sum_{|k| > N} \sum_{t\in \TTT} \sum_{s\in \SSS_k}
g_{ts}^*(f_n)  \int_\R  (\Del_{I_k} g ) (y-t) e^{2\pi i s y} \pphi(x-y) dy
\end{equation}
holds with the interval $I_k = [-k \del - l, -k \del + l]$.

We now invoke 
Rubio de Francia's Littlewood--Paley inequality \cite{Rub85},
which can be stated as follows: Let $2 \le p < \infty$,
and suppose that $\{I_k\}$ is a sequence of disjoint intervals on $\R$. 
Then for every $f \in L^p(\R)$, 
\begin{equation}
\label{eq:rdf.1}
\Big\| \Big(\sum_{k} 
|\Del_{I_k} f |^2 \Big)^{1/2} \Big\|_p \le C_p \|f\|_p,
\end{equation}
where $C_p$ is a constant that  depends only on $p$.
(In fact, we will only need a weaker version, for intervals $I_k$ of
equal lengths,  which was proved earlier by L.\ Carleson). 

In our case, we have
$I_k = [-k \del - l, -k \del + l]$, $k \in \Z$, 
which do not form a system of
disjoint intervals. However,
our system $\{I_k\}$ can be partitioned into a finite number
of subsystems, such that each subsystem is composed of disjoint intervals.
Hence, by applying \eqref{eq:rdf.1} to each subsystem, 
 it is easy to infer that 
\begin{equation}
\label{eq:rdf.5}
\Big\| \Big(\sum_{k \in\Z} | \Del_{I_k} g|^2\Big)^{1/2} \Big\|_p 
\end{equation}
is finite. In turn, this implies
(e.g.\ by an application of the dominated convergence theorem)
 that  if we fix an arbitrarily small 
 $\eta> 0 $, then for all sufficiently large $N$ (depending on 
 $g$, $p$, $\del$, $\pphi$ and $\eta$,
but not depending on $n$) we have
\begin{equation}
\label{eq:6.56}
\Big\| \Big(\sum_{|k| > N} | \Del_{I_k}  g |^2\Big)^{1/2} \Big\|_p   \le \eta.
\end{equation}

Now, using the expansion \eqref{eq:b2.d.2}, we can estimate $|B^N_n(x)|$ as 
\begin{align}
&|B^N_n (x) | \le  \sum_{|k| > N} \sum_{t\in \TTT} \sum_{s\in \SSS_k}
|g_{ts}^*(f_n)|  \int_\R  | (\Del_{I_k} g ) (y-t) | \cdot | \pphi(x-y) | dy \label{eq:c1.h.1} \\
& =  \int_\R \Big( \sum_{|k| > N}  | (\Del_{I_k} g ) (y) |  \cdot 
\Big(\sum_{t\in\TTT} \sum_{s\in\SSS_k} 
|g_{ts}^*(f_n)|\cdot |\pphi(x-y-t)|\Big) \Big) dy \label{eq:c1.h.2}  \\
&\le \int_\R \Big( \sum_{|k|>N}  | (\Del_{I_k} g ) (y) |^2 \Big)^{1/2}
 \Big( \sum_{k \in \Z} \Big( \sum_{t\in\TTT} \sum_{s\in\SSS_k}
 |g_{ts}^*(f_n)|\cdot |\pphi(x-y-t)| \Big)^2 \Big)^{1/2} dy. \label{eq:c1.h.3}
\end{align}
As a consequence, if we denote
\begin{equation}
\label{eq:c2.0}
\Phi_{n}(y) =  \Big( \sum_{k\in\Z} 
\Big( \sum_{t\in\TTT} \sum_{s\in\SSS_k} 
|g_{ts}^*(f_n)|\cdot |\pphi(y-t)| \Big)^2 \Big)^{1/2},
\end{equation}
then we get from \eqref{eq:c1.h.1}--\eqref{eq:c1.h.3} 
and the estimate \eqref{eq:6.56} that
\begin{align}
\| B^N_n \|_p & \le \Big\| \Big(\sum_{|k| > N} |  \Del_{I_k}  g |^2\Big)^{1/2} 
\ast \Phi_{n} \Big\|_p \label{eq:c2.1} \\
& \le \Big\| \Big(\sum_{|k|>N} | \Del_{I_k} g|^2\Big)^{1/2}
 \Big\|_p \cdot \| \Phi_{n} \|_1 \le \eta \| \Phi_{n} \|_1. \label{eq:c2.2} 
\end{align}
Hence, it  only remains to show that there is a finite constant $C > 0$
which does not depend on $n$, such that 
\begin{equation}\label{eq:6.62}
 \| \Phi_{n} \|_{1}  \le C, \quad n=1,2,3,\dots.
\end{equation}
Indeed, in this case we may choose $\eta=C^{-1} \eps$, 
and the estimate \eqref{eq:c2.1}--\eqref{eq:c2.2} 
 implies that for all sufficiently large $N$ 
 (independently of  $n$) we have 
 $\|B^N_n\|_p \le \eps$.

The required bound \eqref{eq:6.62} is relatively easy to prove.
Due to \eqref{eq:t.a}, \eqref{eq:c2.0}  we have
\begin{equation}
\label{eq:c3.1}
\Phi_{n}(y) \le \Big( \sum_{k\in\Z} \Big( \sum_{t\in\TTT} \sum_{s\in\SSS_k} 
|g_{ts}^*(f_n)| \Big)^2 \Big)^{1/2}  \sup_{|t| \le a} |\pphi(y-t)|.
\end{equation}
The first term on the right hand side 
does not exceed $M$,
due to our estimate \eqref{eq:keyest} and since
$\|f_n\|_p\le 1$. The second term on the right hand side
of \eqref{eq:c3.1} satisfies
\begin{align}
 \int_{\R}  \Big( \sup_{|t| \le a} |\pphi(y-t)| \Big) dy < + \infty,
\end{align}
since the Schwartz function $\pphi$ has fast decay. 
These observations establish the required estimate \eqref{eq:6.62}, 
and this completes the proof of \thmref{thm:framesp>2}.
\qed

\emph{Remark}. We could have just as well 
considered $U_\pphi$ as an operator
$L^p(\R) \to L^q(\R)$, for any exponent $q$
in the range  $p \le q \le \infty$. The only
necessary change is to replace the inequality 
$\|\pphi \ast \psi \|_p \le \|\pphi\|_1 \|\psi\|_p$
used in the estimates \eqref{eq:an.c.1} and
\eqref{eq:c2.1}--\eqref{eq:c2.2} 
with the more general Young's convolution inequality 
$\|\pphi \ast \psi \|_q \le \|\pphi\|_r \|\psi\|_p$,
where the exponent $r$ is determined by the relation
$p^{-1} + r^{-1} = 1 + q^{-1}$.
In this case, we get that for any weakly null sequence 
$\{ f_n \}$ in $L^p(\R)$, we have
$\|U_\pphi f_n\|_q \to 0$, which again gives
a contradiction.

% =======================================

\section{Balian--Low type result: Nonexistence of nice window functions}
\label{sec:nicewind}

We finally turn to the question of existence of 
unconditional Gabor frames in $L^p(\R)$ such that
the window function $g$ is  ``nice'', i.e.\ enjoying 
certain smoothness and decay properties. 
In this section we prove \thmref{thm:nicegen},
which is a Balian--Low type result giving a negative answer in this direction.

Recall that the \emph{Wiener amalgam space} $W(L^\infty, \ell^1)$, that
we  shortly denote by $W$, consists of all measurable
functions $g$ on $\R$ satisfying
\begin{equation}
\label{eq:am.def.2}
\|g\|_W =\sum_{k\in\Z} \|g|_{[k,k+1]}\|_{\infty} < \infty,
\end{equation}
see e.g.\ \cite[Section 11.4]{Hei11}.
In particular, any function $g \in W$ is bounded and vanishes at infinity.
Moreover, $g \in L^p(\R)$ for any $1 \le p < \infty$.

Let $p \ne 2$, and suppose that $g$ is a continuous function in $W$.
Our goal is to show that if $\Lam \sbt \TTT \times \R$
for some uniformly discrete set $\TTT\subset \R$, then 
there do not exist any 
coefficient functionals $\{ g^*_{ts} \}$ in $(L^p(\R))^*$,
such that the system $\{(e_s \tau_t g, g^*_{ts})\}$, $(t,s) \in \Lam$, 
forms an unconditional Schauder frame in  $L^p(\R)$.

As we mentioned in the introduction, the assumption of continuity of the function $g$
can be relaxed. Our proof below works for any function $g\in W$
which can be uniformly approximated by compactly supported
step functions.

We also note that the proof works in the more general case where 
$\TTT \sbt \R$ is a set of bounded density, i.e.\ a finite union of uniformly discrete sets.

The exponent $p$ will be assumed below to satisfy $p \ne 2$. 
However, note that for $1<p<2$ we proved a stronger
result, namely, \thmref{thm:gabframesp<2}.

The remainder of the section is devoted to the proof of \thmref{thm:nicegen}.

\subsection{}
We start with a simple lemma which shows that the
elements of the space $W$ satisfy a somewhat more general
condition that the definition \eqref{eq:am.def.2}.

\begin{lem}
\label{lem:amalgam}
Let $I \subset \R$ be a bounded interval, 
and let $\TTT\subset\R$ be a finite union of uniformly discrete sets.
Then for any  $g \in W$, 
\begin{equation}
\label{eq:am.w.1}
\sum_{t\in \TTT} \|(\ts_tg)|_{I}\|_{\infty} \le A \|g\|_W,
\end{equation}
where $A$ is a constant that depends on $I$ and $\TTT$,
but does not depend on $g$.
\end{lem} 

The reader should notice that the definition 
\eqref{eq:am.def.2} of the space $W$ corresponds to
the special case  where $I = [0,1]$ and $\TTT = \Z$
in the condition \eqref{eq:am.w.1}.

Since the proof of \lemref{lem:amalgam} is straightforward, we omit it.

\subsection{}
As before, by adding extra elements to the system 
with zeros as coefficient functionals,
we may assume that $\Lambda = \TTT \times \SSS$ 
where $\TTT\subset\R$ is a  uniformly discrete set (or, more generally, a finite union of uniformly discrete sets), and $\SSS\subset\R$ is an arbitrary countable set.
Assume that $\{(e_s \ts_tg, g_{ts}^*)\}$,
$(t,s) \in \TTT \times  \SSS$, is a $K$-unconditional Schauder frame in the
space $L^p(\R)$. Then every function $f\in L^p(\R)$ admits an expansion
\begin{equation}
\label{eq:f.ex.bl}
f = \sum_{t\in \TTT} \sum_{s\in \SSS} g_{ts}^*(f) e_s \ts_t g,
\end{equation}
where the series is unconditionally convergent in $L^p(\R)$.

Since we assume that
$g$ can be uniformly approximated by compactly supported step functions,
and $g$ is a nonzero function, we can choose and fix 
an interval $I\subset\R$ and a constant $c > 0$ such that $|g(x)|\ge c$
for every  $x\in I$.

For each $t \in T$, we then denote
\begin{equation}
\label{eq:ut.th}
(u_{t} f)(x) = \sum_{s\in \SSS} g_{ts}^*(f) e_s(x), \quad x \in I.
\end{equation}

\begin{lem}
\label{lem:estim_exp}
The series \eqref{eq:ut.th} converges unconditionally in the space $L^p(I)$, and
\begin{equation}
\label{eq:estim_exp}
\| u_t f\|_{L^p(I)}\le M \|f\|_{L^p(\R)}
\end{equation}
where $M$ is a constant which depends neither on $t$ nor on $f$.
\end{lem}

\begin{proof}
Denote $\theta_{ts} = e^{- 2 \pi i st}$. The series
$(v_t f)(x) = \sum_{s\in \SSS} \theta_{ts} g_{ts}^*(f) e_s (x) g(x-t)$
converges unconditionally in $L^p(\R)$, and satisfies
$\|v_t f\|_p \le K \|f\|_p$. Let us consider an operator
$w_t: L^p(\R) \to L^p(I)$ defined by
$(w_t f)(x) = f(x+t) / g(x)$, $x \in I$.
Then  $w_t$ is a bounded operator,
$\|w_t\| \le c^{-1}$, and $w_t$ maps the terms of the series
$v_t f$  to the corresponding terms of the series $u_t f$.
This implies the assertion of the lemma
with $M = K c^{-1}$.
\end{proof}

\emph{Remark}. A similar argument shows
that the series \eqref{eq:ut.th} also
converges unconditionally in $L^p$ on any translate
of  $I$. Hence, the series converges unconditionally locally in $L^p$.

\subsection{}
Let $\eta = (4M)^{-1}$, where $M$ is the constant from \eqref{eq:estim_exp}. 
It follows from \lemref{lem:amalgam} that  we can
choose a sufficiently large $r  >0$ such that
\begin{equation}
\label{eq:am.w.2}
\sum_{|t| > r} \| (\ts_tg)|_I \|_{\infty} < \eta,
\end{equation}
where, here and below, $t$ runs through the elements of the set $\TTT$.  We have
\begin{equation}
 \Big\|   \sum_{|t|>r}  \sum_{s\in \SSS}  
g_{ts}^*(f) e_s  \ts_t g \Big\|_{L^p(I)}
\le  \sum_{|t|>r} \|  (u_t f) \cdot (\ts_t g) \|_{L^p(I)} 
\le \sum_{|t| > r} \| (\ts_tg)|_I \|_{\infty}  \|u_t f \|_{L^p(I)}.
\end{equation}
Using the estimates \eqref{eq:estim_exp} and \eqref{eq:am.w.2},
this implies that
\begin{equation}
\label{eq:bl.est.1}
 \Big\|   \sum_{|t|>r}  \sum_{s\in \SSS}  g_{ts}^*(f) e_s  \ts_t g \Big\|_{L^p(I)}
\le M \eta \|f\|_p = (1/4) \|f\|_p
\end{equation}
for every $f \in L^p(\R)$.

\subsection{}
Since $\TTT$ is a finite union of uniformly discrete sets, the
set $\TTT \cap [-r,r]$ is finite. The number of
elements of this set will be denoted by $N$. Let $\eps = (4MN)^{-1}$,
then by our assumption, we can 
find a compactly supported step function $h$ such that 
\begin{equation}
\label{eq:cr.5} 
\|g - h \|_{L^\infty(\R)} < \eps.
\end{equation}
We have
\begin{equation}
\label{eq:cr.1} 
 \Big\|   \sum_{|t| \le r}  \sum_{s\in \SSS}  
g_{ts}^*(f) e_s  \ts_t (g -  h) \Big\|_{L^p(I)} 
\le  \sum_{|t| \le r} \|  (u_t f) \cdot \ts_t (g - h) \|_{L^p(I)}.
\end{equation}
Due to \eqref{eq:estim_exp} and \eqref{eq:cr.5}, 
each summand on the right hand side of \eqref{eq:cr.1} 
does not exceed $M \eps$, and there are $N$ such 
summands. Hence,
\begin{equation}
\label{eq:bl.est.2}
 \Big\|   \sum_{|t| \le r}  \sum_{s\in \SSS}  
g_{ts}^*(f) e_s  \ts_t (g -  h) \Big\|_{L^p(I)} 
\le MN \eps \|f\|_p = (1/4) \|f\|_p,
\end{equation}
which again holds for every $f \in L^p(\R)$.

\subsection{}
Now, for each $f \in L^p(\R)$, we denote
\begin{equation}
\label{eq:sf.op.bl}
Sf =   \sum_{|t| \le r}  \sum_{s\in \SSS}  
g_{ts}^*(f) e_s   \ts_t h.
\end{equation}
This series converges unconditionally in the space $L^p(I)$,
due to \lemref{lem:estim_exp} and the fact that
$h$ is a bounded function.
Let us estimate the quantity
$\|f - Sf\|_{L^p(I)}$.
From the expansions \eqref{eq:f.ex.bl},  \eqref{eq:sf.op.bl}, 
and using the estimates
\eqref{eq:bl.est.1}, \eqref{eq:bl.est.2}, we get
\begin{align}
& \|f - S f \|_{L^p(I)} \le 
 \Big\|   \sum_{|t|>r}  \sum_{s\in \SSS}  
g_{ts}^*(f) e_s  \ts_t g \Big\|_{L^p(I)} \label{eq:sf.op.es.1} \\
&    \qquad  + 
\Big\|  \sum_{|t| \le r}  \sum_{s\in \SSS}  
g_{ts}^*(f) e_s  \ts_t (g -  h) \Big\|_{L^p(I)} 
\le  (1/2) \|f\|_p. \label{eq:sf.op.es.2} 
\end{align}

\subsection{}
The elements $\{\tau_t h\}$, $t \in \TTT \cap [-r,r]$, form a finite
system of  compactly supported step functions.
Hence, we can find a subinterval $J \sbt I$ such that each
one of these functions is constant on $J$, say,
$(\tau_t h)(x) = c_t$ for every $x \in J$, 
where $c_t$ is a constant.

Let us consider the space $L^p(J)$ as a closed subspace of
$L^p(\R)$, by extending each  function $f \in L^p(J)$ to
be zero outside the interval $J$.
For each $s \in \SSS$, we denote
\begin{equation}
\label{eq:bl.hts.1}
h_{s}^*(f) = \sum_{|t| \le r} c_t g_{ts}^*(f).
\end{equation}
Then, if we restrict the series
\eqref{eq:sf.op.bl} to the interval $J$, we obtain
\begin{equation}
(Sf)(x) = \sum_{s\in \SSS}   h_{s}^*(f) e_s(x), \quad x \in J,
\end{equation}
where the latter series converges unconditionally in $L^p(J)$.
Since $\{h_{s}^*\}$, $s \in \SSS$, are continuous linear functionals
on the space $L^p(J)$, the estimate \eqref{eq:sf.op.es.1}--\eqref{eq:sf.op.es.2}
therefore implies that the system
$\{ (e_s, h_s^*) \}$, $s \in \SSS$, forms
an unconditional approximate Schauder frame in $L^p(J)$.
In turn, by an application of  \lemref{lem:approxframe},
we conclude that there exist functionals $e_s^* \in (L^p(J))^*$ 
such that the system
$\{ (e_s, e_s^*) \}$, $s \in \SSS$, forms
an unconditional Schauder frame in the space $L^p(J)$.

The last conclusion now  provides  us with the desired contradiction: we 
recently proved \cite[Theorem 1.2]{LT25b} that if $p \neq 2$,
then there does not exist 
any unconditional Schauder frame consisting of exponential
functions in the space $L^p(\Om)$, for any open set
$\Om \sbt \R$ of finite measure.
This contradiction concludes the proof of \thmref{thm:nicegen}.
\qed

% =======================================

\end{document}